\documentclass{article}
\usepackage[a4paper]{geometry}
\usepackage{amsmath,amssymb,amsthm}
\usepackage{pict2e}
\usepackage{dsfont}
\usepackage{graphicx}
\usepackage{overpic}

\newcommand{\R}{\mathbb R}
\newcommand{\N}{\mathbb N}
\newcommand{\C}{\mathbb C}
\newcommand{\B}{\mathcal{B}}
\newcommand{\A}{\mathcal{A}}
\newcommand{\F}{\mathcal{F}}
\newcommand{\T}{\mathcal{T}}
\newcommand{\al}{\alpha}
\newcommand{\g}{\gamma}
\newcommand{\G}{\Gamma}
\newcommand{\la}{\lambda}
\newcommand{\io}{\iota}
\newcommand{\s}{\sigma}
\newcommand{\Om}{\Omega}
\newcommand{\om}{\omega}
\newcommand{\e}{\eta}
\newcommand{\ep}{\varepsilon}
\newcommand{\tep}{{\tilde{\varepsilon}}}
\newcommand{\del}{\partial}
\newcommand{\ti}[1]{\tilde{#1}}
\newcommand{\VG}{V\!G}
\newcommand{\1}{\mathds{1}}
\newcommand{\sle}{{\rm SLE}}
\newcommand{\2}{\tfrac{1}{2}}
\newcommand{\h}{\frac{1}{2}}
\newcommand{\hc}[1]{\boxminus #1}

\newcommand{\af}{{\hat{\alpha}_4}}
\newcommand{\at}{{\check{\alpha}_2}}

\theoremstyle{plain}
 \newtheorem{thm}{Theorem}[section]
 \newtheorem{lem}[thm]{Lemma}
 \newtheorem{corol}[thm]{Corollary}
\theoremstyle{remark}
 \newtheorem{rem}[thm]{Remark}
 
\DeclareMathOperator{\Var}{Var}
\DeclareMathOperator{\dist}{dist}
\DeclareMathOperator{\IM}{Im}
\DeclareMathOperator{\id}{id}

\title{Singularity of Nearcritical Percolation Exploration Paths}
\author{Simon Aumann\\[0.5ex] \textit{\small{Mathematisches Institut, Ludwig-Maximilians-Universit\"at M\"unchen}}\\[-0.5ex] \textit{\small{Theresienstr.\ 39, D-80333 M\"unchen, Germany}}\\[-0.5ex] \small{aumann@math.lmu.de}}

\begin{document}

\maketitle

\begin{abstract}
We show that the laws of scaling limits of nearcritical percolation exploration paths with different parameters are singular with respect to each other. This generalises a result of Nolin and Werner, using a similar technique. As a corollary, the singularity can even be detected from an infinitesimal initial segment. Moreover, nearcritical scaling limits of exploration paths are mutually singular under scaling maps. 
\end{abstract}

\begin{small}
\end{small}

\section{Introduction}
One break-through of the mathematical theory of two-dimensional percolation was in 2001, when Smirnov proved the conformal invariance of the scaling limit of critical percolation interfaces on the triangular lattice. This paved the way for describing this limit by a Schramm-Loewner-Evolution and for determining various crossing probabilities. Thus nowadays the scaling limit of critical percolation is quite well understood. But there are also nearcritical scaling limits. These are obtained by choosing the probability for a site being open depending on the mesh size slightly different from the critical one, but converging to it in a well-chosen speed. These nearcritical limits are by far not as well understood as the critical ones. Garban, Pete and Schramm showed in \cite{gps13} that, in the quad-crossing space, there indeed exist nearcritical limits, not only limit points. But we do not use this fact, since we are interested in the exploration paths. For that, only the existence of limit points, not of a 
limit, is yet established.
 
Nolin and Werner showed in \cite{nw9} that every nearcritical scaling limit point of exploration paths is singular with respect to an $\sle_6$ curve, i.e.\ to the critical limit. In the present note, we enhance this result by showing that two different nearcritical scaling limits are singular with respect to each other (Theorem~\ref{mainthm}). It is even possible to detect the singularity by looking at an infinitesimal initial segment of the exploration path (Corollary~\ref{cor:infini}). Applying the main result to conformal maps, we obtain that nearcritical scaling limits are in general not conformally invariant or absolutely continuous. In fact, under scaling maps, they are mutually singular (Corollary~\ref{cor:conf}).

Interestingly, the proof of Nolin and Werner can be extended to our result. But one has to be careful. In fact, we also give a more detailed and self-contained version of their proof. Nevertheless, some modifications and slightly different approaches are needed. In particular, the non-existence of an analogue to Cardy's formula requires some work.  Namely, we need the fact that the probability of crossing a quad with fractal boundary can be well approximated using rather weak approximations to the quad (Lemma~\ref{lem:kern}).

The organisation of this note is as follows. In Section~\ref{sec:not} we introduce precisely the model and state the main theorem, which will be proved in Section~\ref{sec:proof}. But before there is an expository section, namely Section~\ref{sec:heu}. There we review some aspects of \cite{nw9} and give some heuristics why the the result of Nolin and Werner as well as our theorem should be true. Finally, in Section~\ref{sec:conf}, we discuss consequences of our result for conformal maps.

\section{Notation and Statement of the Main Theorem} \label{sec:not}
Let us start with the basic definitions and notations. Let $H_r:=\{z\in\C: |z|<r,\IM(z)>0\}$ be the upper half circle with radius $r>0$. We work on the hexagonal lattice with mesh size $\e>0$. Let $H^\e_r$ be all hexagons of size $\e$ which are entirely contained in $\overline{H_r}$.

We consider face percolation in $H^\e_r$ with different parameters $p^\mu$ and $p^\la$. Thereto let $\mu,\la\in\R$  and $\mu_\e,\la_\e\in\R$, $\e>0$, such that $\mu_\e\to\mu$ and $\la_\e\to\la$ as $\e\to0$. Each hexagon is independently of the others blue (open) with probability
\begin{equation*}
 p^\io = p^\io_\e \,=\, \h + \io_\e \cdot\frac{\e^2}{\al_4^\e(1)} 
\end{equation*}
and otherwise yellow (closed), where we choose $\io\in\{\mu,\la\}$ depending on the desired parameter. Here $\al_4^\e(R)$ is the probability that there exists four arms of alternating colours up to (Euclidean) distance $R$ in critical site percolation on the triangular lattice with mesh size $\e$. Smirnov and Werner showed in \cite[Theorem 4]{sw1} that $\al_4^\e(1)=\e^{\frac{5}{4}+o(1)}$ as $\e\to0$. Therefore (or by using  the five arm exponent) it follows that  $p^\io\to\2$ as $\e\to0$. As we are interested in that limit, we may hence choose $\e$ small enough such that $p^\io\in(0,1)$. Thus we work on the families of probability spaces
$$ \Big(\, \Om_\e:=\{\text{blue,yellow}\}^{H^\e_r}, \quad \mathfrak{P}(\Om_\e), \quad P^\io_\e := \bigotimes_{H^\e_r} \big( p^\io\delta_\text{blue} + (1-p^\io)\delta_\text{yellow}\big) \,\Big)_{\e>0}$$
with $\io\in\{\mu,\la\}$. The choice of $p^\io$ ensures that we are still in the critical window, but obtain scaling limits different from the critical one (if $\io\ne0$, of course). This follows from Kesten's scaling relations and can explicitly be deduced from \cite[Proposition 4]{nw9} together with \cite[Proposition 32]{n7}, for example.

If we colour the negative real axis blue and the positive axis yellow, then there is a unique path, called \emph{exploration path}, on the hexagonal lattice starting at the origin and stopping $\e$-close to the upper boundary of $H_r$, which has blue hexagons to the left and yellow hexagons to the right. Let us denote this path by the random variable 
$$ \g_\e:\, (\Om_\e,\mathfrak{P}(\Om_\e)) \rightarrow (\mathcal{S}_r,\B(\mathcal{S}_r)) \,,$$
where $\mathcal{S}_r$ (with Borel-$\s$-algebra $\B(\mathcal{S}_r)$ induced by the metric below) is the space of curves in $\overline{H_r}$, i.e. equivalence classes of continuous functions $[0,1]\to \overline{H_r}$. Two such functions $f,g$ represent the same curve if and only if $f=g\circ\phi$ for some increasing bijection $\phi:[0,1]\to[0,1]$. We introduce a topology on $\mathcal{S}_r$ via the metric 
$$\dist(f,g):=\inf_\phi\max_{t\in[0,1]} |f(t)-g\circ\phi(t)|\,$$
where the infimum is taken over all increasing bijections $\phi:[0,1]\to[0,1]$. Then $\mathcal{S}_r$ is a complete separable space. Let
$$ \G^\io_\e \, := \, \g_\e(P^\io_\e)$$
denote the law of $\g_\e$ under $P^\io_\e$, for $\e>0$ and $\io\in\{\mu,\la\}$. Using a technique developed by Aizenman and Burchard in \cite{ab99}, Nolin and Werner showed in \cite[Proposition 1]{nw9} that the family $(\G^\io_\e)_{\e>0}$ is tight, i.e. for each sequence $\e_k$ there is a subsequence $\e_{k_l}$ such that $\G^\io_{\e_{k_l}}$ converges weakly.

For the statement of the main theorem we need, in contrast to Nolin and Werner, a result using the Quad-Crossing Topology introduced by Schramm and Smirnov in \cite{ss11}. Therefore we review that concept very briefly. For a much more detailed account one should consult \cite[p.\ 1778f]{ss11}. Let $D$ be a domain. A \emph{quad} $q$ in $D$ is a topological quadrilateral, i.e.\ a homeomorphism $q:[0,1]^2\to q([0,1]^2)\subset D$. Let $\mathcal{Q}_D$ be the set of all quads in $D$. A quad $q$ is crossed by a percolation configuration, if the union of all blue (topologically closed) hexagons contains a connected closed subset of $\overline{q}:=q([0,1]^2)$ which intersects both opposite sides $\del_0q:=q(\{0\}\times[0,1])$ and $\del_2q:=q(\{1\}\times[0,1])$. This event is denoted by $\hc{q}\subset \Om_\e$.  We will further need the notations $\del_1q:=q([0,1]\times\{0\})$ and $\del_3q:=q([0,1]\times\{1\})$ for the other two sides of the quad. Moreover, let $q^\circ:=q((0,1)^2)$ be the interior and $\del q$ be the 
whole boundary of $q$.

Using a partial order on $\mathcal{Q}_D$ induced by crossings, one can define the set $\mathcal{H}_D$ of all closed lower sets $S\subset\mathcal{Q}_D$. Schramm and Smirnov constructed a topology on $\mathcal{H}_D$, namely the Quad-Crossing-Topology. For our purposes the following facts are enough. There is a random variable $cr: \Om_\e\to\mathcal{H}_D$ which assigns each percolation configuration the set of all crossed quads. Thus each probability measure on $\Om_\e$ induces a probability measure on $\mathcal{H}_D$. Moreover, the space of all probability measures on $\mathcal{H}_D$ is tight (\cite[Corollary 1.15]{ss11}). Finally, if $\mathbb{P}$ is any limit point of the measures $cr(P_\e^\mu)$, $\e>0$, then $\mathbb{P}[\del cr(\hc{q})]=0$ for every quad $q\in\mathcal{Q}_D$ (\cite[Lemma 5.1]{ss11}). Therefore there exists a sequence $(\e_k)_{k\in\N}$ with $\lim_{k\to\infty}\e_k=0$ such that $P^\mu_{\e_k}[\hc{q}]$ converges as $k\to\infty$ for all quads $q\in\mathcal{Q}_D$.

Now we are ready to state the main theorem of the present note.
\begin{thm} \label{mainthm}
 Let $\mu<\la$ be real numbers, $\mu_\e\to\mu$, $\la_\e\to\la$ and $r>0$. Let further $(\e_k)_{k\in\N}$ be a sequence converging to zero such that $P^\mu_{\e_k}[\hc{q}]$ converges for all quads $q\in\mathcal{Q}_{H_r}$ and such that $\G^\mu_{\e_k} \to \G^\mu$ and $\G^\la_{\e_k} \to \G^\la$ weakly for some measures $\G^\mu$ and $\G^\la$ on $(\mathcal{S}_r,\B(\mathcal{S}_r))$ as $k\to\infty$.
 
 Then the probability measures $\G^\mu$ and $\G^\la$ are singular with respect to each other.
\end{thm}
$\G^\mu$ and $\G^\la$ are distributions of the scaling limits of the discrete exploration paths (in the limit point sense). Let us remark that \cite[Proposition 6]{nw9} is included in this theorem as the special case $\mu=\mu_\e=0$. In that case the hypothesis on the quad crossing probabilities is always fulfilled since it follows from Cardy's formula. But in our case, we unfortunately do not have any analogue; that is the reason for the additional condition.

The theorem also holds if $\mu>\la$, i.e.\ if the condition on the quad crossing probabilities holds for the larger value. In that case quite a few inequality signs have to be switched. Thus for better readability, we restrict ourselves to the case $\mu<\la$.

Actually we do not need to look at the whole exploration path to detect the singularity. In fact, it is enough to look at an infinitesimal initial segment as the following corollary shows. We consider the space $(\mathcal{S}_1,\B(\mathcal{S}_1))$ of curves in $H_1$. Let
$$ \tau_n(\g) := \inf\{t\ge0: |\g(t)|=\tfrac1n\}$$
be the first exit time of $H_{\frac1n}$ and
$$ \A_n := \s(\id[0,\tau_n], \id(0)=0) $$
be the $\s$-algebra generated by curves starting at the origin until exiting $H_{\frac1n}$, $n\in\N$. Then $\A_n$, $n\in\N$, is decreasing. Let
$$ \A := \bigcap_{n\in\N} \A_n $$
be their tail-$\s$-algebra, the $\s$-algebra of infinitesimal initial segments of paths starting at the origin. With that notation, Theorem~\ref{mainthm} implies
\begin{corol} \label{cor:infini}
 Under the conditions of Theorem~\ref{mainthm}, the laws $\G^\mu$ and $\G^\la$ restricted to $\A$ are singular with respect to each other.
\end{corol}
\begin{proof}
 By Theorem~\ref{mainthm} applied to $r=\frac1n$, there are sets $A_n\in\A_n$ with $\G^\mu[A_n]=0$ and $\G^\la[A_n]=1$. We set
 $$ A_* := \bigcup_{m\ge1} \bigcap_{n\ge m} A_n \,.$$
 Then $A_*\in\A$. Since countable unions or intersection of sets of probability zero respectively one have probability zero respectively one, it follows that $\G^\mu[A_*]=0$ and $\G^\la[A_*]=1$, which proves the corollary.
\end{proof}

We conjecture that Theorem~\ref{mainthm} and its corollary also hold on other lattices. In fact, if we can apply RSW techniques, most elements of the proof work. We need the separation lemmas and other results of \cite{n7}, which are delicate consequences of RSW (\cite[Theorem 2]{n7}). Thus they remain true on other lattices, cf.\ \cite[Section 8.1]{n7}. We further need the following bounds on arm events. Let $\al^\e_2(\rho,R)$ and $\al^\e_4(\rho,R)$ be the probabilities of the events that at critical percolation with mesh size $\e$ there exist two respectively four arms of alternating colours inside an annulus with radii $\rho$ and $R$ (i.e., in particular, $\al^\e_4(\e,R)=\al^\e_4(R)$). We need that there are ``exponents'' $\af,\at>0$ and constants $c,c'>0$ such that
$$ \al^\e_2(\rho,R) \ge c (\rho/R)^\at \qquad\text{and}\qquad \al^\e_4(\rho,R) \le c' (\rho/R)^\af $$
for all $0<\e\le\rho\le R$ and such that
\begin{equation} \label{eq:inexp}
 2\af-\at > 2 \,. 
\end{equation}
Since the two arm exponent in the half plane exists and is 1 as a consequence of RSW (see \cite[Theorem 23]{n7}, for instance), it follows that we can choose  $\at\le1<2$, which we also need. While the analogues to \cite[Proposition 13]{n7} and \cite[Theorem 10]{n7} yield the existence of such exponents also for other lattices, inequality (\ref{eq:inexp}) is yet proven only for site percolation on the triangular lattice (or equivalently, face percolation on the hexagonal lattice). Indeed, we can choose $\at=\frac14-\beta$ and $\af=\frac54+\beta$ for any $\beta>0$ there. Since the former inequality is the only needed special property of the triangular lattice, we choose to write up the proof with the exponents $\at$ and $\af$ and not with the explicit values. Hence the results can immediately be enhanced to other lattices as soon as inequality (\ref{eq:inexp}) is established.

\section{Heuristics} \label{sec:heu}
This section is of expository nature and therefore not rigorous. First we review some aspects of \cite{nw9}. Then we give a heuristic explanation why a nearcritical scaling limit should be singular with respect to the critical or to another nearcritical scaling limit. These heuristics could in fact also be seen as an outline of the proof. Formally, this section is not needed for the remainder of the article.

Let us recall some of our notation: $P^\io_\e$ denotes the probability measure of nearcritical percolation with parameter $\io\in\R$, i.e. a site is open with probability
$$p^\io \,=\, \h + \io \cdot\frac{\e^2}{\al_4^\e(\e,1)}  \,.$$
Moreover, the random variable $\g_\e$ denotes the exploration path and $\G^\io_\e$ its law under $P^\io_\e$.

A basic concept of nearcritical percolation is the introduction of a characteristic length.  Below that length, the Russo-Seymour-Welsh Theory (RSW) is still valid. This means that the probability that a set is crossed by the percolation configuration (in some specific way) does only depend on the shape of the set, but not on its size -- as long as this size is below the characteristic length. In the set-up considered in this article, the mesh size of the lattice and the nearcritical probabilities are chosen such that the characteristic length is of order one. Thus RSW techniques are applicable.

The first result of Nolin and Werner \cite[Proposition~1]{nw9} shows tightness of the laws of the exploration paths. We shortly outline their proof. It is an application of \cite[Theorem~1.2]{ab99}.  Let us denote the annulus around $x$ with radii $\rho<R$ by $A(x,\rho,R)$. RSW considerations imply that there exist some constants $c,\alpha>0$ such that
$$ P^\io_\e\big[\g_\e \text{ crosses } A(x,\rho,R)] \le c(\rho/R)^\alpha $$
uniformly for all $\e\le\rho\le R$. Using the BK Inequality, it follows that, for all $k\in\N$,
$$ P^\io_\e\big[\g_\e \text{ crosses } A(x,\rho,R) \text{ $k$ times}] \le c_k(\rho/R)^{\alpha k} \,.$$
Therefore the hypothesis of \cite[Theorem~1.2]{ab99} is fulfilled and tightness follows. This means that for each sequence $\e_k$ there is a subsequence $\e_{k_l}$ such that $\G^\io_{\e_{k_l}}$ converges weakly.

Nolin and Werner also determined the Hausdorff dimension of any sub-sequential scaling limit of the critical and nearcritical exploration paths. It is $7/4$ in both cases, see \cite[Proposition~3]{nw9}. The proof is based on RSW techniques and the knowledge of the two-arm exponent of critical percolation.

The perhaps most important result of \cite{nw9} is Proposition~6. It states that the law of any nearcritical sub-sequential limit is singular with respect to the law of an $\sle_6$ curve, which is the critical limit. As already mentioned, we enhance this result in the present note and show that $\G^\mu\bot\G^\la$, where $\G^\io$ is a limit point of $\G^\io_\e$, $\io\in\{\mu,\la\}$. In the following, we heuristically argue why these theorems hold.

Let us consider an equilateral triangle $\Delta$ of size $\delta$. The scale $\delta$ should be an intermediate one, i.e $\e\ll\delta\ll1$. We assume that the exploration path $\g_\e$ entered the triangle somewhere in the middle of the triangle's bottom line and is at time $\s$ somewhere in the middle of the triangle. If that is the case, we say that the triangle is \emph{good} for $\g_\e$. 
\begin{figure}[ht]
\begin{minipage}{0.47\textwidth} \begin{center}
\begin{picture}(50,44)
 \thicklines
 \put(0,0){\line(1,0){50}}
 \put(0,0){\line(500,866){25}}
 \put(50,0){\line(-500,866){25}}
 \qbezier(26,-1)(22,2)(23.5,6)
 \qbezier(23.5,6)(24,12)(24.5,6)
 \qbezier(24.5,6)(25,0)(25.5,6)
 \qbezier(25.5,6)(26,9)(25,12)
 \put(25,12){\circle*{1}} \put(18.5,12.8){$\g(\s)$}
 \put(25,0){\circle*{1}}
 \put(50,0){\circle*{1}}
 \qbezier[12](25,12)(30,27.5)(35,12)
 \qbezier[7](35,12)(37,6)(38,0)
 \put(36,10){\huge ?}
\end{picture} 
\caption{A (maybe very) good triangle \label{fig:dreieck1}} \end{center}
\end{minipage}
\hfill
\begin{minipage}{0.47\textwidth} \begin{center}
\begin{picture}(50,44)
 \thicklines
 \put(0,0){\line(1,0){50}}
 \put(0,0){\line(500,866){25}}
 \put(50,0){\line(-500,866){25}}
 \qbezier(26,-1)(22,2)(23.5,6)
 \qbezier(23.5,6)(24,12)(24.5,6)
 \qbezier(24.5,6)(25,0)(25.5,6)
 \qbezier(25.5,6)(26,9)(25,12)
 \put(25,12){\circle*{1}} 
 \put(25,0){\circle*{1}}
 \put(50,0){\circle*{1}}
 \put(30,20){\circle{1}}
 \qbezier[16](30,19.5)(34,8)(26,8)
 \qbezier[16](30,20.5)(28,28)(20,34.6)
 \thinlines
 \qbezier(29.5,20)(5,13)(23.5,6)
 \qbezier(30.5,20)(40,15)(42,0)
 
 \put(27,14){\large $\g$}
 \qbezier[12](25,12)(30,27.5)(35,12)
 \qbezier[7](35,12)(37,6)(38,0)
\end{picture}
\caption{A pivotal site with four arms \label{fig:piv}}  \end{center}
\end{minipage}
\end{figure}
We even look at the following stronger event: Conditionally on $\g_\e[0,\s]$, we ask whether $\g_\e$ exists the triangle on the right part of the bottom line. In that case we call the triangle even \emph{very good} for $\g_\e$. This events are schematically drawn in Figure~\ref{fig:dreieck1}.

We estimate the difference of the probability of being very good, conditionally on $\g_\e[0,\s]$, under $P^\la_\e$ and under $P^\mu_\e$. Thereto we use the standard monotone coupling of percolation with different parameters $p\in[0,1]$ (for all hexagons not discovered by $\g_\e[0,\s]$). Thus the set $\om(p)$ of blue hexagons at level $p$ increases. If a good triangle $\Delta$ is very good for $\g_\e(\om(p^\la))$, but not for $\g_\e(\om(p^\mu))$, then there exists a site $x$ in the triangle which is pivotal for some crossing event and switched from yellow to blue, cf.\ Figure~\ref{fig:piv}. It is pivotal, iff there are four arms of alternating colours from $x$ to some described parts of the boundary. Therefore we conclude
\begin{align*}
 & P^\la_\e\big[\Delta\text{ is very good for }\g_\e\mid\g_\e[0,\s]\big] -
 P^\mu_\e\big[\Delta\text{ is very good for }\g_\e\mid\g_\e[0,\s]\big] \\
 =\;\;&
 P\big[\Delta\text{ is very good for }\g_\e(\om(p^\la))\text{ but not for }\g_\e(\om(p^\mu))\mid\g_\e[0,\s]\big] \\
 \approx\;\;  &
 P\big[\exists\, x\in\Delta\setminus\g_\e[0,\s] :\text{ four arms from }x\text{ to }\partial\Delta, x\text{ switched between }p^\mu\text{ and }p^\la\big]
\end{align*}
Since the crossing event is increasing, the latter event can happen only for one $x$ inside the triangle. Since there are around $(\delta/\e)^2$ sites inside the triangle, we conclude
\begin{align*}
  &P\big[\exists\, x\in\Delta\setminus\g_\e[0,\s] :\text{ four arms from }x\text{ to }\partial\Delta,\, x\text{ switched between }p^\mu\text{ and }p^\la\big] \\
  \approx\;\;&
  (\delta/\e)^2 \al_4^\e(\e,\delta) (p^\la-p^\mu) \\
  =\;\;&
  (\delta/\e)^2 \al_4^\e(\e,\delta)\, (\la-\mu)\,\e^2/\al_4^\e(\e,1) \,,
\end{align*}
where we used
$$ p^\la-p^\mu = \2+\la \e^2 /\al_4^\e(\e,1) - \2 - \mu \e^2 /\al_4^\e(\e,1) = (\la-\mu)\,\e^2/\al_4^\e(\e,1)$$
in the last step. Now $\la-\mu\asymp1$ and quasi-multiplicativity, i.e.\ $\al_4^\e(\e,1)\asymp\al_4^\e(\e,\delta)\al_4^\e(\delta,1)$, and finally $\al_4^\e(\delta,1)\to\delta^{5/4}$ yield
$$ (\delta/\e)^2 \al_4^\e(\e,\delta)\, (\la-\mu)\,\e^2/\al_4^\e(\e,1) 
 \,\approx\,
 \delta^2 / \al_4^\e(\delta,1) \,\approx\, \delta^{3/4} \,.$$
Thus we established the estimate
$$ P^\la_\e\big[\Delta\text{ very good for $\g$}\mid\Delta\text{ good for $\g$}\big] - P^\mu_\e\big[\Delta\text{ very good for $\g$}\mid\Delta\text{ good for $\g$}\big] \,\approx\, \delta^{3/4} $$
for every triangle $\Delta$ of scale $\delta$.

We will use this estimate to evaluate the expectation of the random variable
$$ Z^\delta(\g) := \#\{\text{very good triangles of scale $\delta$ for $\g$}\} - E^\mu[\#\{\text{very good triangles of scale $\delta$ for $\g$}\}]\,.$$
Since the Hausdorff dimension of the exploration path is $7/4$, it touches approximately $\delta^{-7/4}$ triangles. By RSW, the number of good triangles is of the same order of magnitude. Therefore we conclude
$$  E^\mu[Z^\delta] =\delta^{-7/4}\cdot0= 0 \qquad\text{and}\qquad E^\la[Z^\delta]\approx \delta^{-7/4}\cdot\delta^{3/4} = \delta^{-1}\,.$$
Though the events being good or very good of different triangles are not independent, we can conclude using a martingale approach that
$$ \Var^\mu[Z^\delta] \le \delta^{-7/4} \qquad\text{and}\qquad \Var^\la[Z^\delta] \le \delta^{-7/4} \,.$$
Now by Chebyshev's inequality, it follows that
$$ P^\mu[Z^\delta > \delta^{-15/16}] \,\le\, \delta^{15/8}\Var^\mu[Z^\delta] \,\le\, \delta^{15/8}\delta^{-7/4}=\delta^{1/8} $$
and
$$ P^\la[Z^\delta < \delta^{-15/16}] \,\approx\, P^\la\big[Z^\delta -E^\la[Z^\delta]< \delta^{-\frac{15}{16}}-\delta^{-1}\big]
\,\le\, \big(\delta^{-15/16}(1-\delta^{-\frac{1}{16}})\big)^{-2}\Var^\la[Z^\delta] \,\le\, \delta^{1/8}\,.$$
Now we choose a sequence of scales $(\delta_n)_n$ such that $\delta_n^{1/8}$ is summable. Then the Lemma of Borel-Cantelli implies
$$ P^\mu[Z^{\delta_n}(\g) > \delta_n^{-15/16}\text{ for infinitely many }n] = 0 $$
and
$$ P^\la[Z^{\delta_n}(\g) < \delta_n^{-15/16}\text{ for infinitely many }n] = 0 \,.$$
As the complements of these events are disjoint, the mutual singularity of $\G^\mu=\g(P^\mu)$ and $\G^\la=\g(P^\la)$ follows.

\section{Proof of the Main Theorem} \label{sec:proof}
We partition the rigorous proof of Theorem~\ref{mainthm} in four subsections. In Section~\ref{ssec:quad} we prove a lemma which is also of independent interest. It states that we can approximate the probability of crossing a quad even if it has fractal boundary and if we use quite weak approximations to it. In Section~\ref{ssec:onetriangle} we look at one mesoscopic triangle, whereas in Section~\ref{ssec:manytriangles} we give estimates for many mesoscopic triangles. Finally, in Section~\ref{ssec:limit}, we consider the continuum limit to conclude the proof of Theorem~\ref{mainthm}.

\subsection{A Quad Crossing Lemma} \label{ssec:quad}

We say that a sequence $(q_n)_{n\in\N}$ of quads converges in the kernel (or Caratheo\-dory) sense to a quad $q$ with respect to some $z_0\in\C$, if 
\begin{itemize}
 \item $z_0\in q_n^\circ$ for all $n\in\N$ and $z_0\in q^\circ$, 
 \item for every $z\in q^\circ$ there exists a neighbourhood of $z$ which is contained in all but finitely many $q_n^\circ$ (and in $q^\circ$),
 \item for each $z\in\del q$ there exist $z_n\in\del q_n$ with $z_n\to z$ and
 \item $q_n(i,j)\to q(i,j)$ for $(i,j)\in\{0,1\}^2$.
\end{itemize}
This is the usual kernel convergence for domains with the additional requirement that the corners of the quads converge. We further need the following condition, which is illustrated in Figure~\ref{fig:quad1}:
\begin{equation} \label{eq:regin}
 \begin{array}{l}
  \forall\, \ep>0 \;\exists\, n_0\in\N \;\forall\, n\ge n_0, i\in\{0,1,2,3\}: \\[0.3ex]
  \qquad U_\ep(\del_iq)\cap\big(\overline{q}\cup U_\ep(\del_{i-1}q)\cup U_\ep(\del_{i+1}q)\big) \text{ contains a path connecting } \\[0.3ex]
  \qquad \del_{i-1}q_n\text{ and }\del_{i+1}q_n \text{ not intersecting }\del_iq_n
 \end{array}
\end{equation}
Here and in the following, $U_\ep(\cdot)$ denotes the  $\ep$-neighbourhood. We use cyclic indexes, i.e. $3+1\equiv0$. The condition demands that $\del_iq_n$ is not close to any other side of $q_n$ or $q$ inside the quad for a long time. Thus inside $\overline{q}$, $\del_iq_n$ is close to $\del_iq$. But note that there may be parts of $\del_iq_n$ far away from $\del_iq$ and even $\del q$ outside $\overline{q}$.
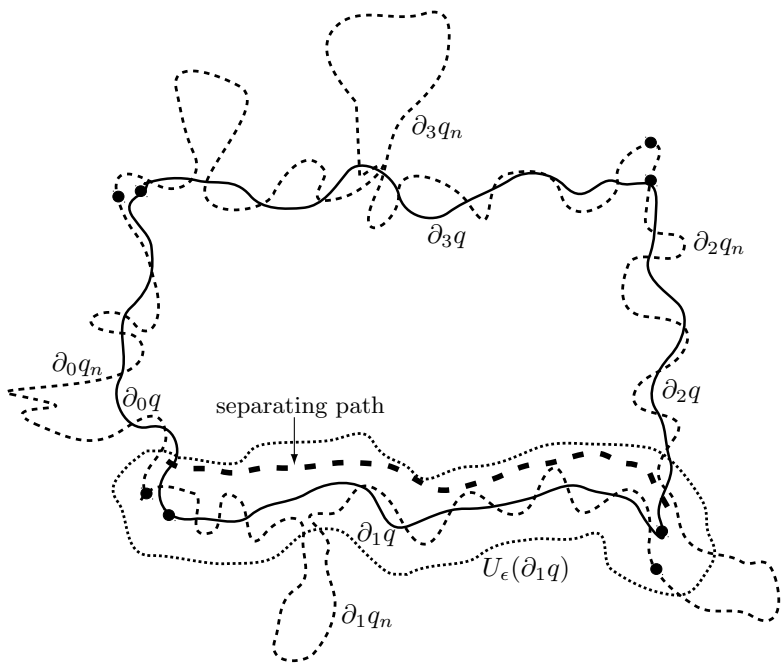
\begin{figure}[ht]
 \begin{center}
   \begin{overpic}[width=0.7\textwidth]{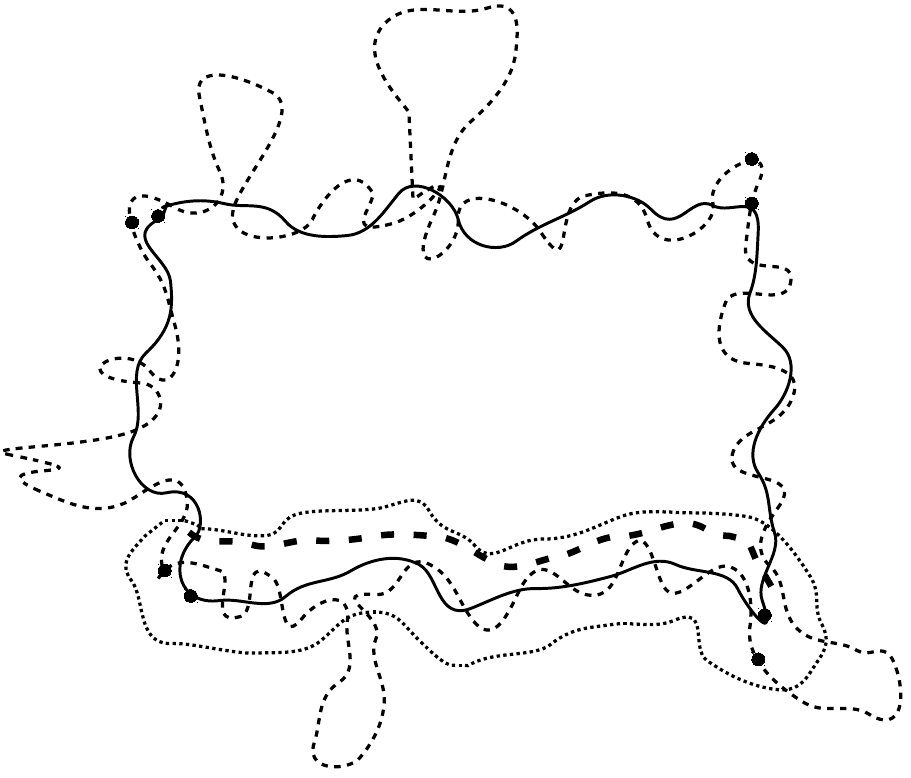} 
    \put(6,38.5){$\del_0q_n$}
    \put(43,6){$\del_1q_n$}
    \put(88,54){$\del_2q_n$}
    \put(52,69){$\del_3q_n$}
    \put(15,34){$\del_0q$}
    \put(45,16.5){$\del_1q$}
    \put(84,35){$\del_2q$}
    \put(54,55){$\del_3q$}
    \put(61,12){$U_\epsilon(\del_1q)$}
    \put(27,33){\small separating path}
    \put(37,32.5){\vector(0,-1){6.2}}
   \end{overpic}
 \end{center}
 \caption{Quads $q$ (solid) and $q_n$ (dashed) satisfying condition (\ref{eq:regin}) with the neighbourhood of $\del_1q$ (fine dotted) and a separating path (strong dashed) \label{fig:quad1}}
\end{figure}

\begin{lem} \label{lem:kern}
 Let some quads $q_n, n\in\N,$ converge in the kernel sense to a quad $q$ as $n\to\infty$ (with respect to some $z_0$). Assume further that condition (\ref{eq:regin}) is fulfilled. Let $P_\e$, $\e>0$, be any (near-)critical probability measures, i.e.\ $P_\e=P_\e^\io$ for any bounded sequence $(\io_\e)_\e\subset\R$.
 
 Then for all $\rho>0$ there exist $n_0\in\N$ and $\e_0>0$ such that for all $n\ge n_0$ and $\e\le\e_0$
 $$ P_\e\big[\hc{q_n}\,\triangle\hc{q}\big] \,\le\, \rho \,,$$
 where $\triangle$ denotes the symmetric difference.
\end{lem}
Let us remark that we do not impose any smoothness conditions on the boundary of the quad. Otherwise, we could just use the $3$-arm-exponent in the half plane. We further remark that the proof relies only on RSW techniques. Thus the lemma is valid on any lattice where RSW works.

In order to prove Lemma~\ref{lem:kern}, we want to apply Lemma A.1 of \cite{ss11}. It states that if two quads differ only at one side by some $\zeta$, then the probability of the  symmetric difference of the corresponding crossing events is small. More precisely, a slightly simplified version reads as follows in our notation. 
\begin{quote} \textit{
 Let $d>0$. There exists a positive function $\Delta(\zeta)$ such that $\Delta(\zeta)\to0$ as $\zeta\to0$ and the following estimates hold. If two quads $q,q'$ of diameter at least $d$ satisfy for some $\zeta<d/2$
 \begin{enumerate}
   \item[(i)] $[\ldots]$ or
   \item[(ii)] $\overline{q'}\subset\overline{q}$, $\del_0q'=\del_0q$, $\del_1q'\subset\del_1q$, $\del_3q'\subset\del_3q$ and each point on $\del_2q'$ can be connected to $\del_2q$ by a path in $\overline{q}$ of diameter at most $\zeta$, or
   \item[(iii)] $\overline{q'}\subset\overline{q}$, $\del_0q'\subset\del_0q$, $\del_1q'=\del_1q$, $\del_2q'\subset\del_2q$ and and each point on $\del_3q'$ can be connected to $\del_3q$ by a path in $\overline{q}$ of diameter at most $\zeta$,
 \end{enumerate}
 then for all $\e<\zeta$
 $$ P_\e\big[\hc{q}\triangle\hc{q'}\big] \,\le\,\Delta(\zeta) \,.$$
 }
\end{quote} 
For the sake of completeness, we shortly outline how one can prove that. Let two quads $q,q'$ satisfy condition (iii). If $\hc{q}\triangle\hc{q'}$ happens, there exists a yellow vertical crossing of $q$ and two blue arms from a disk of radius $\zeta$ to $\del_0q'$ and $\del_2q'$. If we condition on the left-most yellow vertical crossing, percolation on the right of it is still unbiased. Therefore we can apply RSW, yielding that the probability of an arm from a disk of radius $\zeta$ to $\del_2q'$ tends to $0$ as $\zeta\to0$, as desired. The details are properly written up in \cite{ss11}.
 
\begin{proof}[Proof of Lemma~\ref{lem:kern}.] 
 First we claim that for each $\ep>0$ there exists an $n_0\in\N$ such that for all $n\ge n_0$ the following holds:
 \begin{itemize}
  \item $|q_n(i,j)-q(i,j)|<\ep$ for each $(i,j)\in\{0,1\}^2$
  \item for any $z\in\del_iq$ there exist $z_n\in\del_iq_n$ with $|z-z_n|<\ep$, $i\in\{0,1,2,3\}$ and
  \item $\overline{q}\setminus U_\ep(\del q) \subset \overline{q_n}$
 \end{itemize}
 Note the uniformity and that $z$ and $z_n$ belong to the same side. Indeed, the first item is obvious from the kernel convergence. The second item can be fulfilled by covering $\del_iq$ with finitely many balls of radius $\ep/2$ (Condition (\ref{eq:regin}) with $\ep/2$ ensures that the $z_n$'s belong to the correct side). Finally, using compactness, a finite sub-cover of the covering of $\overline{q}\setminus U_\ep(\del q)$ by the neighbourhoods used in the definition of the kernel convergence yields the third item.
 
 Let $\ep>0$. We will specify $\ep$ depending on $\rho$ later on. Let $n\ge n_0$, where $n_0$ is associated to $\ep$  such that the claim and condition (\ref{eq:regin}) hold with this $n_0$. We need a further scale $\tep=\ep^\al\gg\ep$ for some $\al>0$ specified below. For $i\in\{0,1,2,3\}$, let $u_i^\tep$ be a closed curve, homeomorphic to a circle, around $\del_iq$, which stays between the $\tep$- and the $2\tep$-neighbourhood of $\del_iq$. We try to avoid that some of the $u_i^\tep$ intersect each other outside the $2\tep$-neighbourhoods of the quad-corners. If this is not possible (for example, when $q$ contains a slit), we treat the affected regions as different.

 We label the corners of the quad $q$ with $a=q(0,1)$, $b=q(0,0)$, $c=q(1,0)$ and $d=q(1,1)$. Now we define some points on the curves $u_i^\tep$ near the corners. Starting at some point of $u_0^\tep$ near $b$ and moving along $u_0^\tep$ outside $q$ (i.e.\ in counter-clockwise direction), let $a_b$ the first hit point of $u_0^\tep\cap u_3^\tep$. Similarly, let $a_d$ the first hit point of $u_0^\tep\cap u_3^\tep$ starting near $d$ and moving along $u_3^\tep$ outside $q$ (i.e.\ now in clockwise direction). Analogously we define the points $b_a$, $b_c$, $c_b$, $c_d$, $d_c$ and $d_a$. The notation should be interpreted as follows: a point $e_f$ (with $e,f\in\{a,b,c,d\}$) is near to the corner $e$, but on the way to $f$ on the curve $u^\tep$ outside $q$. These definitions are illustrated in Figure~\ref{fig:quad22}.
 \begin{figure}[ht]
 \begin{center}
   \begin{overpic}[width=0.7\textwidth]{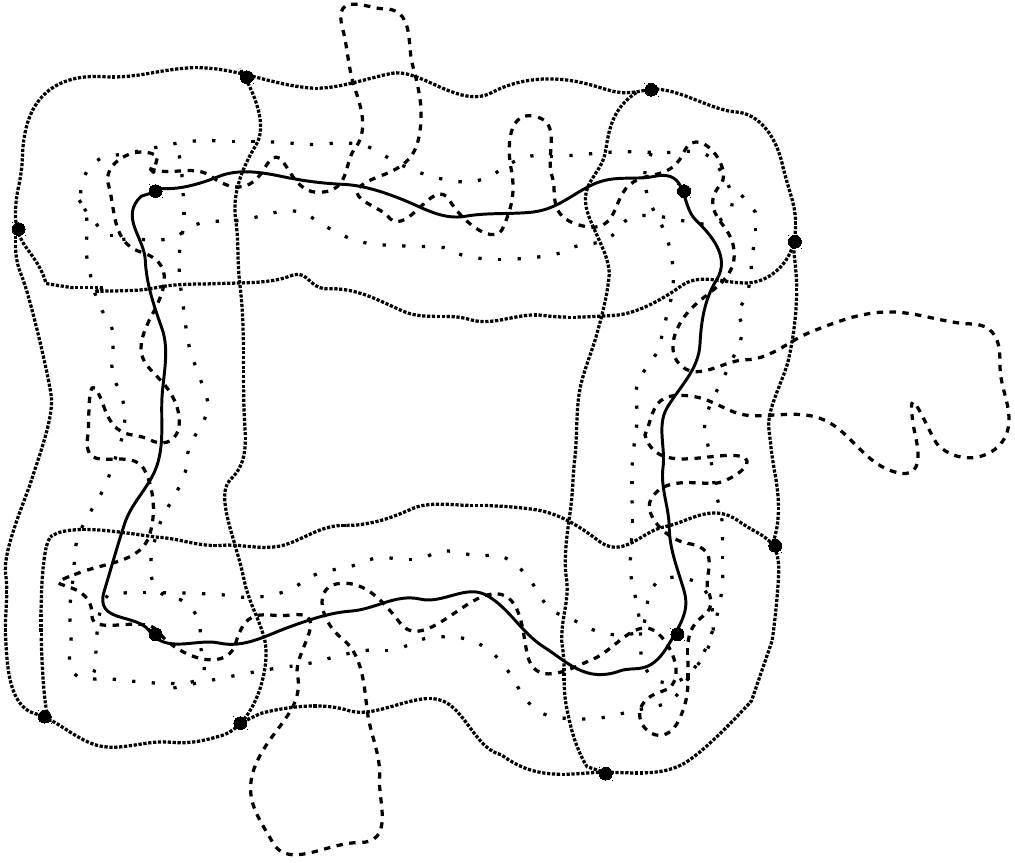}
    \put(13,66){$a$}
    \put(-2,63){$a_b$}
    \put(24,79){$a_d$}
    \put(13,20){$b$}
    \put(4,11){$b_a$}
    \put(22.5,10.5){$b_c$}
    \put(65,23){$c$}
    \put(59,6){$c_b$}
    \put(78,28){$c_d$}
    \put(67,67){$d$}
    \put(79,61.5){$d_c$}
    \put(64,77){$d_a$}
    \put(0,44){$u_0^\tep$}
    \put(43,10){$u_1^\tep$}
    \put(77,38){$u_2^\tep$}
    \put(46,78){$u_3^\tep$}
   \end{overpic}
 \end{center}
 \caption{Quads $q$ (solid) and $q_n$ (dashed) with the $\ep$-neighbourhood of $q$ (wide dotted), the curves $u_i^\tep$ (fine dotted) of $q$ and the marked points \label{fig:quad22}}
\end{figure}
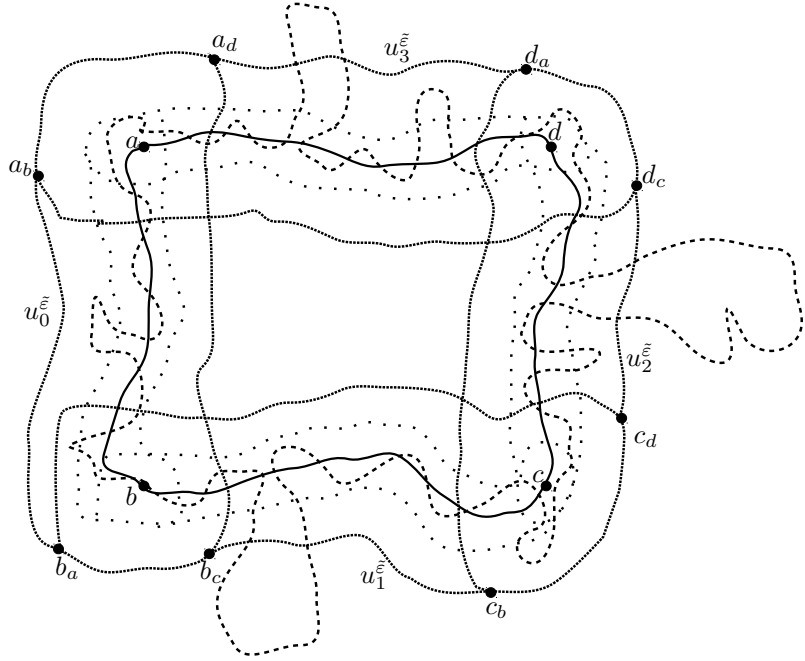
  
 We use these points and curves to define the following quads. They are schematically drawn in Figure~\ref{fig:quads} below. We define the quads by giving the corners and the sides. We do not specify the parametrisations, since they are irrelevant. Let $q^0$ be defined by the corners $a_d$, $b_c$, $c_b$ and $d_a$ with the following sides: Let $\del_0q^0$ consist of the part of $u_0^\tep$ between $a_d$ and $b_c$ which intersects $\bar{q}$. The side $\del_1q^0$ consists of the part of $u_1^\tep$ between $b_c$ and $c_b$ which stays outside $\bar{q}$. The side $\del_2q^0$ shall consist of the part of $u_2^\tep$ between $c_b$ and $d_a$ which intersects $\bar{q}$. And finally, let $\del_3q^0$ consist of the part of $u_3^\tep$ between $d_a$ and $a_d$ which stays outside $\bar{q}$. We abbreviate this definition by
 \begin{eqnarray*}
  q^0 &=& \big[a_d \text{ --i-- }\, b_c \text{ --o-- } c_b \text{ --i-- } d_a \text{ --o--}\big]
 \end{eqnarray*}
 Here we give the corners and the sides between them. An  ``--o--''  indicates that the corresponding side consists of the part of $u_i^\tep$ between the given corners which stays \textbf{o}utside $\bar{q}$, whereas an ``--i--'' denotes that the part of $u_i^\tep$ which \textbf{i}ntersects $\bar{q}$ is used. With this notation we further define the quads
 \begin{eqnarray*}
  q^1 &=& \big[a_d \text{ --i-- } b_c \text{ --o-- } c_d \text{ --o-- } d_c \text{ --o--}\big] \\
  q^2 &=& \big[a_b \text{ --o-- } b_a \text{ --o-- } c_d \text{ --o-- } d_c \text{ --o--}\big] \\
  q^3 &=& \big[a_b \text{ --o-- } b_a \text{ --o-- } c_d \text{ --o-- } d_c \text{ --i--}\big] \\
  q^4 &=& \big[a_b \text{ --o-- } b_a \text{ --i-- } c_d \text{ --o-- } d_c \text{ --i--}\big]  
 \end{eqnarray*}
 which are schematically drawn in Figure~\ref{fig:quads}.
 \begin{figure}[ht]
  \begin{picture}(100,20)(-5,-4)
   \thinlines
   \put(0,12){\line(1,0){16}}
   \put(0,9){\line(1,0){16}}
   \put(0,3){\line(1,0){16}}
   \put(0,0){\line(1,0){16}}
   \put(0,0){\line(0,1){12}}
   \put(3,0){\line(0,1){12}}
   \put(13,0){\line(0,1){12}}
   \put(16,0){\line(0,1){12}}
   \put(0,9){\circle{1.3}}
   \put(0,3){\circle{1.3}}
   \put(3,12){\circle*{1.3}}
   \put(3,0){\circle*{1.3}}
   \put(13,12){\circle*{1.3}}
   \put(13,0){\circle*{1.3}}
   \put(16,9){\circle{1.3}}
   \put(16,3){\circle{1.3}}   
   \put(2,13.8){$a_d$}
   \put(12,13.8){$d_a$}
   \put(2,-3.3){$b_c$}
   \put(12,-3.3){$c_b$}
   \put(-3.8,8.5){$a_b$}
   \put(-3.8,2.5){$b_a$}
   \put(17.1,8.5){$d_c$}
   \put(17.1,2.5){$c_d$}
   
   \put(7,-4.5){$q^0$}
   \put(29,-4.5){$q^1$}
   \put(49,-4.5){$q^2$}
   \put(69,-4.5){$q^3$}
   \put(89,-4.5){$q^4$}
   
   \put(22,12){\line(1,0){16}}
   \put(22,9){\line(1,0){16}}
   \put(22,3){\line(1,0){16}}
   \put(22,0){\line(1,0){16}}
   \put(22,0){\line(0,1){12}}
   \put(25,0){\line(0,1){12}}
   \put(35,0){\line(0,1){12}}
   \put(38,0){\line(0,1){12}}
   
   \put(42,12){\line(1,0){16}}
   \put(42,9){\line(1,0){16}}
   \put(42,3){\line(1,0){16}}
   \put(42,0){\line(1,0){16}}
   \put(42,0){\line(0,1){12}}
   \put(45,0){\line(0,1){12}}
   \put(55,0){\line(0,1){12}}
   \put(58,0){\line(0,1){12}}
   
   \put(62,12){\line(1,0){16}}
   \put(62,9){\line(1,0){16}}
   \put(62,3){\line(1,0){16}}
   \put(62,0){\line(1,0){16}}
   \put(62,0){\line(0,1){12}}
   \put(65,0){\line(0,1){12}}
   \put(75,0){\line(0,1){12}}
   \put(78,0){\line(0,1){12}}
   
   \put(82,12){\line(1,0){16}}
   \put(82,9){\line(1,0){16}}
   \put(82,3){\line(1,0){16}}
   \put(82,0){\line(1,0){16}}
   \put(82,0){\line(0,1){12}}
   \put(85,0){\line(0,1){12}}
   \put(95,0){\line(0,1){12}}
   \put(98,0){\line(0,1){12}}

   \linethickness{0.5mm}
   \put(3,0){\line(0,1){12}}
   \put(13,0){\line(0,1){12}}
   \put(3,0){\line(1,0){10}}
   \put(3,12){\line(1,0){10}}
   
   \put(25,0){\line(1,0){13}}
   \put(25,12){\line(1,0){13}}
   \put(25,0){\line(0,1){12}}
   \put(38,0){\line(0,1){12}}
   
   \put(42,0){\line(1,0){16}}
   \put(42,12){\line(1,0){16}}
   \put(42,0){\line(0,1){12}}
   \put(58,0){\line(0,1){12}}
      
   \put(62,0){\line(1,0){16}}
   \put(62,9){\line(1,0){16}}
   \put(62,0){\line(0,1){9}}
   \put(78,0){\line(0,1){9}} 
      
   \put(82,3){\line(1,0){16}}
   \put(82,3){\line(0,1){6}}
   \put(82,9){\line(1,0){16}}
   \put(98,3){\line(0,1){6}}
   
   \put(25,0){\circle*{1.3}}
   \put(25,12){\circle*{1.3}}
   \put(38,3){\circle*{1.3}}
   \put(38,9){\circle*{1.3}}
   
   \put(42,3){\circle*{1.3}}
   \put(42,9){\circle*{1.3}}
   \put(58,3){\circle*{1.3}}
   \put(58,9){\circle*{1.3}}
   
   \put(62,3){\circle*{1.3}}
   \put(62,9){\circle*{1.3}}
   \put(78,3){\circle*{1.3}}
   \put(78,9){\circle*{1.3}}
   
   \put(82,3){\circle*{1.3}}
   \put(82,9){\circle*{1.3}}
   \put(98,3){\circle*{1.3}}
   \put(98,9){\circle*{1.3}}
  \end{picture}
  \caption{Schematic drawing of the quads $q^0$, $q^1$, $q^2$, $q^3$ and $q^4$ \label{fig:quads}} 
 \end{figure}
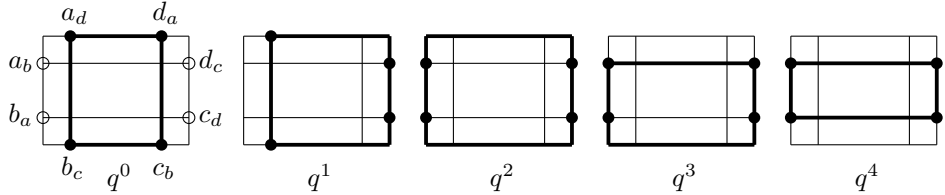

 Then
 $$ \hc{q^0}\triangle\hc{q^4} \,\subseteq\, \bigcup_{i=0}^3 \hc{q^i}\triangle\hc{q^{i+1}} $$
 and each pair $(q^i,q^{i+1})$, $i=0,1,2,3$, satisfies condition (ii) or (iii) of Lemma A.1 in \cite{ss11} as cited above with $\zeta=4\tep$ (for $i=1,3$, the sides $\del_0$ and $\del_2$ as well as the sides $\del_1$ and $\del_3$ have to be interchanged). We conclude for $\e<\tep$
 $$ P_\e\big[\hc{q^0}\triangle\hc{q^4}\big] \,\le\, f(\tep) $$
 for some function $f$ with $f(\tep)\to0$ as $\tep\to0$.
 
 Now we want to link the previous observation to the event of interest. By the construction of the quads $q^0$ and $q^4$, every crossing of $q^4$ contains a crossing of $q$ and every crossing of $q$ contains one of $q^0$, i.e.\ $\hc{q^4}\subseteq\hc{q}\subseteq\hc{q^0}$. This statement is only almost true, if we consider $q_n$ instead of $q$, since $q_n$ may have excursions outside $U_\tep(\bar{q})$, i.e.\ in general $\hc{q^4}\not\subseteq\hc{q_n}\not\subseteq\hc{q^0}$. But as $\del q_n$ will come $4\ep$-close to itself after leaving $U_\ep(\del q)$, we can control the events $\hc{q_n}\setminus\hc{q^0}$ and $\hc{q^4}\setminus\hc{q_n}$, as follows. Mind that we now use the $\ep$-neighbourhoods. We will need the distance between $\ep$ and $\tep$ to control some arm events below.
 
 Let us cover $U_\ep(\del q)$ with finitely many balls of radius $\ep$ centred at points $z_j$, $j\in J$. We need at most $c\ep^{-2}$ many balls, with some numerical constant $c>0$. Assume that there exists $x\in\del_iq_n\setminus U_\ep(\overline{q})$, i.e.\ some part of $\del_iq_n$ is far away from $\overline{q}$. Then we claim that there exist $j\in J$ and $x_1,x_2\in U_{2\ep}(z_j)\cap\del_iq_n$ such that $x$ lies in between $x_1$ and $x_2$ on $\del_iq_n$. Indeed, let $\del_iq_n|^1$ respectively $\del_iq_n|^2$ be the part of $\del_iq_n\cap U_\ep(\del_iq)$ before respectively after $x$, and let $U^k:=U_\ep(\del_iq_n|^k)$, $k\in\{1,2\}$. Then $\del_iq \subseteq U^1\cup U^2$, since for all $z\in\del_iq$ there exists $z_n\in\del_iq_n$ with $|z_n-z|<\ep$ (second item above), i.e.\ $z_n\in\del_iq_n|^1\cup\del_iq_n|^2$. Thus $U^1\cap U^2\ne\emptyset$. Therefore there exists $j\in J$ with $U_\ep(z_j)\cap U^1\cap U^2\ne\emptyset$. We conclude that there are $ y_k\in U_\ep(z_j)\cap U^k$ and $x_k\in\del_iq_n|^k$ with 
$|y_k-x_k|<\ep$, which implies $|x_k-z_j|<2\ep$, $k\in\{1,2\}$, as claimed.
 
 Now if $\hc{q_n}\setminus\hc{q^0}$ happens, each crossing of $q_n$ must leave $q^0$ between $b_c$ and $c_b$ or between $d_a$ and $a_d$. By the geometry of $q_n$, explained in the claim above, the crossing is forced to re-enter some ball $B_{2\ep}(z_j)$ with $z_j\in\overline{q^0}$ after leaving $q_0$ (at least $\tep$ away from $\del q$). Furthermore, it must reach the paths whose existence is postulated in condition (\ref{eq:regin}) for $i=0,2$. Thus it reaches the $\ep$-neighbourhoods of $\del_0q$ and $\del_2q$, which are of distance at least $\tep-2\ep$ of the ball. Thus the crossing induces four blue arms inside the annulus centred at $z_j$ with radii $2\ep$ and $\tep-2\ep$. Moreover, there must exist two yellow arms inside this annulus preventing $q^0$ being crossed. The event $\hc{q^4}\setminus\hc{q_n}$ is treated similarly, or by duality, considering a yellow vertical crossing of $q_n$ which does not induce a vertical crossing of $q^4$. Therefore, we conclude
 $$ \big(\hc{q_n}\setminus\hc{q^0}\big) \cup \big(\hc{q^4}\setminus\hc{q_n}\big) \subseteq \bigcup_{j\in J}A_6(z_j,2\ep,\tep-2\ep)\,, $$
 where $A_6(z,\varrho,R)$ denotes the event that there exist six arms, not all of them of the same colour, inside the annulus centred at $z$ of radii $\varrho$ and $R$. By standard RSW techniques, we have for $\e<\varrho$
 $$ P_\e\big[A_6(z,\varrho,R)\big] \le (\varrho/R)^{2+\nu} $$
 for some $\nu>0$ (i.e.\ the polychromatic 6-arm-exponent is larger than $2$). Recall that $\tep=\ep^\al$ for some $\al>0$. Therefore $\tep-2\ep\ge\h\ep^\al$ for small $\al$. It follows that
 $$ P_\e\Big[\bigcup_{j\in J}A_6(z_j,2\ep,\tep-2\ep)\Big] \,\le\, c\ep^{-2}\cdot (\tfrac{2\ep}{\tep-2\ep})^{2+\nu} \,\le\,c \ep^{\nu-2\al-\nu\al}\,,$$
 which tends to zero as $\ep\to0$ for sufficiently small $\al>0$.
 
 Summing up, we have
 $$ \hc{q_n}\triangle\hc{q} \;\subseteq\; \big(\hc{q^0}\triangle\hc{q^4}\big)\,\cup\, \bigcup_{j\in J}A_6(z_j,2\ep,\tep-2\ep)$$
 and therefore for $\e<\ep$
 $$ P_\e\big[\hc{q_n}\triangle\hc{q}\big] \,\le\,  \tilde{f}(\ep) $$
 for some function $\tilde{f}$ with $\tilde{f}(\ep)\to0$ as $\ep\to0$.
 
 To conclude the proof, given $\rho>0$, we choose $\ep>0$ such that $\tilde{f}(\ep)\le\rho$, $\e_0=\2\ep$ and $n_0\in\N$ associated to $\ep$ as above.
\end{proof}
 
\begin{rem}
 Just convergence in the kernel sense is not enough, as the following counterexample shows. Let $q$ be the quad $q:[0,1]^2\to[0,1]^2$, $q(z)=z$, and let quads $q_n$ be given by
 $$ \overline{q_n}:=[0,1]^2\setminus(\tfrac{1}{n},1]\times(\tfrac{1}{n},\tfrac{2}{n}), \qquad \del_iq_n=\del_iq,\quad i\in\{0,1,3\},$$
 and $\del_2q_n$ consisting of the boundary part between $(1,0)$ and $(1,1)$.
 
 Then $q_n$ converge in the kernel sense to $q$. But if $P_\e^{0.5}$ denotes the critical percolation measure, then
 $$ P_\e^{0.5}(\hc{q}) = \tfrac{1}{2},$$
 whereas
 $$ P_\e^{0.5}(\hc{q_n}) \to 1 $$ 
 as $\e\to0$ with $\e\asymp 1/n$. RSW yields the last assertion considering concentric (quarter-)annuli around $(0,0)$ with radii $2/n\cdot2^k$ and $2/n\cdot2^{k+1}$, $0\le k \le c\log n$. Thus a condition like (\ref{eq:regin}) is necessary for Lemma~\ref{lem:kern}.
\end{rem}

\subsection{One Mesoscopic Triangle} \label{ssec:onetriangle}

Now we begin with the proof of Theorem~\ref{mainthm}. Using the same basic ideas, we more or less follow the set-up of the proof of \cite[Proposition 6]{nw9}. But now and then we take slightly different approaches for various reasons. In particular, we work longer with the discrete exploration paths.

Let us fix a sequence $(\e_k)_{k\in\N}$ fulfilling the hypothesis of the theorem. As explained before stating the theorem, such a sequence does exist. In the following, we omit the subscript $k$ of $\e_k$ and simply write $\e$ for an element of the chosen sequence. The limit $\e\to0$ is always to be understood along the sequence $(\e_k)_k$.
 
First we need some definitions. Consider a small equilateral triangle $t$ of size $\delta$ which is contained in $H_r$. The size $\delta$ shall be some mesoscopic size, intermediate between the mesh size $\e$ and the size $r$ of the domain. 

According to Figure~\ref{fig:1}, we define the open rectangle $r=r(t)$ to be the whole dotted area, the closed segments $l=l(t)$, $m=m(t)$ and $b=b(t)$ to be the lower, the middle respectively the upper ``line" of $r$ as well as the smaller triangle $t'$ just like in \cite[p.\ 814]{nw9}, to which we also refer for exact definitions. But note that the exact definitions are not that important for the proof.
\begin{figure}[ht]
\begin{minipage}{0.47\textwidth} \begin{center}
\begin{picture}(50,44)
 \thicklines
 \put(0,0){\line(1,0){50}}
 \put(0,0){\line(500,866){25}}
 \put(50,0){\line(-500,866){25}}
 \put(22,6){\line(1,0){6}}
 \put(22,12){\line(1,0){6}}
 \thinlines
 \put(22,0){\line(0,1){12}}
 \put(28,0){\line(0,1){12}}
 \put(5,6){\line(1,0){40}} 
 \put(5,6){\line(500,866){20}} 
 \put(45,6){\line(-500,866){20}}
 \multiput(22.5,0.5)(0,0.5){23}{\multiput(0,0)(0.5,0){11}{\circle*{0.1}}}
 \put(3,1){$t$}
 \put(8,7){$t'$}
 \put(19,8){$r$}	\put(20.8,8.7){\vector(1,0){5}}
 \put(29.2,1){$l$}	\put(29,1.8){\vector(-3,-1){4.3}}
 \put(29.2,7){$m$}	\put(29,7.8){\vector(-3,-1){4.3}}
 \put(29.2,13){$b$}	\put(29,13.8){\vector(-3,-1){4.3}}
\end{picture}
\caption{Definition of $r,m,b,t'$ \label{fig:1}} \end{center}
\end{minipage}
\hfill
\begin{minipage}{0.47\textwidth} \begin{center}
\begin{picture}(50,44)
 \thicklines
 \put(0,0){\line(1,0){50}}
 \put(0,0){\line(500,866){25}}
 \put(50,0){\line(-500,866){25}} 
 \thinlines
 \put(22,12){\line(1,0){6}}
 \put(22,0){\line(0,1){12}}
 \put(28,0){\line(0,1){12}}
 \put(5,6){\line(1,0){40}} 
 \put(5,6){\line(500,866){20}} 
 \put(45,6){\line(-500,866){20}} 
 \put(3,1){$t$}
 \put(45,6){\circle*{1}} \put(44,2.5){$a^0$}
 \put(23.5,6){\circle*{1}} \put(16.5,2.2){$a^2$} \put(19.6,3.8){\vector(2,1){3.4}}
 \put(26.5,6){\circle*{1}} \put(30.5,2.2){$a^1$} \put(30.4,3.8){\vector(-2,1){3.4}}
 \put(25,12){\circle*{1}} \put(22,16.5){$\g(\s)$} \put(25,15.8){\vector(0,-1){3.1}}
 \qbezier(23.5,6)(24,12)(24.5,6)
 \qbezier(24.5,6)(25,0)(25.5,6)
 \qbezier(25.5,6)(26,9)(25,12)
 \qbezier(26.5,6)(28,10)(25,12)
 \put(35,6.5){$\del^0$}
 \put(29.5,25){$\del^1$}
 \put(17.8,25){$\del^1$}
 \put(13,6.5){$\del^1$}
 \put(18,11){$\del^2$} \put(20.2,11){\vector(2,-1){3.5}} 
 \put(31,11){$\del^3$} \put(30.8,11){\vector(-2,-1){3.5}} 
\end{picture}
\caption{Definition of $\g(\s),a^i,\del^i$ \label{fig:2}} \end{center}
\end{minipage}
\end{figure}

Given a curve $\g\in\mathcal{S}_r$, let $\s=\s(t,\g)$ be its first hitting time of $t\setminus r$ or the first hitting time of $l$ after hitting $m$, whatever happens first. If $\g(\s)\in b$ we say that the triangle $t$ is \emph{good} for the curve $\g$. Let us denote this event by $G(t,\g)$. 

If a triangle $t$ is good for a curve $\g$, we define the following. Let $a^0=a^0(t)$ be the right corner of $t'$, $a^1=a^1(t,\g)$ be the right-most point and $a^2=a^2(t,\g)$ the left-most point on $m\cap\g[0,\s]$. We further define the set $d=d(t,\g)$ as the union of the  connected component of $t'\setminus\g[0,\s]$ which has the top boundaries of $t'$ on its boundary, and the components of $r\setminus(t'\cup\g[0,\s])$ which touch the former component between $a^2$ and $a^1$. Then $d(t,\g)$ is a simply connected set whose boundary consists of $\del t'\setminus (a^2,a^1)$ and some points of $\g[0,\s]$. We partition its boundary as follows. Let $\del^0(t,\g)$ be the part of the boundary between $a^1$ and $a^0$, $\del^1(t,\g)$ the part between $a^0$ and $a^2$, $\del^2(t,\g)$ the prime ends between $a^2$ and $\g(\s)$ and finally $\del^3(t,\g)$ the prime ends between $\g(\s)$ and $a^1$ (all in counter-clockwise direction). With these boundary parts, one can consider $d(t,\g)$ as some quad. These definitions are 
illustrated in Figure~\ref{fig:2}. Note that they depend on the curve $\g$ only up to time $\s$. 

Now we define the event $\VG(\g,t)$ that the triangle $t$ is \emph{very good} for $\g$: it holds if $t$ is good for $\g$ and if, after time $\s$, $\g$ hits $\del^0(t,\g)$ before $\del^1(t,\g)$. Note that all these definitions are analogous to \cite[p.\ 814]{nw9}. We only decreased the indices of $\del^i$ to be consistent with the quad notation introduced above. We further enlarged the set $d$ a little bit to ensure the observation in the next paragraph.

When we apply these definitions to the discrete exploration paths $\g_\e$, we adjust them to the discrete setting: All sets shall be unions of hexagons, a point is considered as a hexagon and $\g_\e[0,\s]$ shall be the exploration path up to time $\s$ together with the touching blue and yellow hexagons. If $t$ is good for $\g_\e$, the event $\VG(t,\g_\e)$ is equivalent to the existence of a blue crossing from $\del^0$ to $\del^2$ inside $d(t,\g_\e)$, i.e.\ to $\hc{d(t,\g_\e)}$. This observation is ensured by the slight enlargement of $d$. Without it, the exploration path could bypass some blue crossings using hexagons below $m$.

In the following lemma we estimate the difference between the $P^\la_\e$- and the $P^\mu_\e$-probability of the event that the exploration path is very good for some triangle conditioned on the path up to time $\s$. We state (and use) this lemma only in the discrete setting. By this means, we avoid having to consider a limit simultaneously in the event and in the conditioning -- which is tricky. Let us recall that $\delta$ is the mesoscopic size of the triangle $t$, that $\g_\e:\Om_\e\to\mathcal{S}_r$ is the exploration path and that $\af$ is the exponent bounding the probability of a four arm event from above.
\begin{lem} \label{lem:evg-disc} 
 The following estimate holds for all very small $\beta>0$ and for all small enough $\delta$ and $\e\ll\delta$ on the event $G(t,\g_\e)$:
 $$ P^\la_\e \big[\VG(t,\g_\e) \mid \g_\e[0,\s]\big] \,-\, P^\mu_\e \big[\VG(t,\g_\e) \mid \g_\e[0,\s]\big]
 \,\ge\, \delta^{2-\af+\beta}\,.$$
\end{lem}
Here and in the following, $\e\ll\delta$ means for all $\e<\e_0$ where $\e_0$ depends on $\delta$. In fact, $\e_0=c\delta$ for some universal constant $c>0$ will be enough.
\begin{proof}
 We follow the corresponding part of the proof of Nolin and Werner, see \cite[p.\ 816]{nw9}. Let $\e>0$ be small. We couple the percolation configurations in a monotone manner such that the set of blue hexagons increases. More precisely, let $\hat{P}$ be the uniform measure on $\hat{\Om}_\e:=[0,1]^{H_r^\e}$, and for $p\in[0,1]$ let the random variable $\om(p): \hat{\Om}_\e\to\Om_\e$ be defined by $(\om(p)(\hat{\om}))_x=\text{blue}$ iff $\hat{\om}_x\le p$ for $x\in H_r^\e$ and $\hat{\om}=(\hat{\om}_x)_x \in \hat{\Om}_\e$.    
   
 Given $\g_\e[0,\s]$ and $G(t,\g_\e)$, the event $\VG(t,\g_\e)$ only depends on the hexagons inside $d(t,\g_\e)$ since it is equivalent to $\hc{d(t,\g_\e)}$. Moreover, given $\g_\e[0,\s]$, percolation inside $d(t,\g_\e)$ is still unbiased, i.e.\ we may use all percolation techniques there, for instance RSW and the separation lemmas.
 
 Suppose now that $t$ is good for $\g_\e$. We conclude 
 $$ P^\la_\e\big[\VG(t,\g_\e) \mid \g_\e[0,\s]\big] - P^\mu_\e\big[\VG(t,\g_\e) \mid \g_\e[0,\s]\big] = \hat{P}[E_\e]\,,$$
 where $E_\e$ is the event that there exists a blue crossing from $\del^0(t,\g_\e)$ to $\del^2(t,\g_\e)$ in $d(t,\g_\e)$ for $\om(p^\la)$, but not for $\om(p^\mu)$. 
 
 In order to prove the proposed estimate, we can restrict ourselves to the following sub-event of $E_\e$. For a hexagon $x$ inside a deterministic rhombus of size $0.1\delta$ inside $d(t,\g_\e)$ (away from the boundary) and for $p\in[p^\mu,p^\la]$, let us consider the event  that $x$ is pivotal for the existence of the desired crossing. In that case there are four arms of alternating colours from $x$ to the boundary of $d(t,\g_\e)$. Its probability is bounded from below by $C\,\al_4^\e(\delta)$ for some constant $C>0$, uniformly in $x$, $p$ and $\e\ll\delta$. This is a consequence of the separation lemmas, RSW and the uniform estimates for arm events, which are still valid in the nearcritical regime (cf.\ e.g.\ \cite{n7}). As the crossing event is increasing, the event that $x$ is pivotal and switched from yellow to blue at $p$ (i.e.\ $\hat{\om}_x=p$), can happen only for one hexagon $x$ and for one $p$. 
 Therefore, the $\hat{P}$-probability that this occurs for some $x$ in the rhombus and for some $p\in[p^\mu,p^\la]$, which is clearly a sub-event of $E_\e$, is larger than 
 $$ C  \al_4^\e(\delta) \, (\tfrac{0.1\delta}{\e})^2 \, (p^\la-p^\mu) \,.$$
 Using
 \begin{equation} \label{eq:pla-pmu}
  p^\la-p^\mu = \2+\la_\e \e^2 /\al_4^\e(1) - \2 - \mu_\e \e^2 /\al_4^\e(1) = (\la_\e-\mu_\e)\,\e^2/\al_4^\e(1)
 \end{equation}
 we estimate
 \begin{eqnarray*}
  \hat{P}[E_\e] 
  &\ge&  C\, \al_4^\e(\delta) (\tfrac{0.1\delta}{\e})^2 \, (p^\la-p^\mu) \\
  &\stackrel{(\ref{eq:pla-pmu})}{=}&
  	 C'\, \delta^2 \e^{-2}\,\al_4^\e(\delta)\,[\al_4^\e(1)]^{-1}\e^2(\la_\e-\mu_\e)  \\
  &\ge&	 C''\, \delta^2 [\al_4^\e(\delta,1)]^{-1} (\la-\mu+o(1))\\
  &\ge&	 \delta^{2-\af+\beta} \,,
 \end{eqnarray*}
 the latter if $\delta$ is small enough, depending on $\beta$, $C''$ and the $o(1)$-term. Quasi-multi\-plica\-tivity yields the last but one line. The lemma follows.
\end{proof}
\begin{rem}
 Using the ratio limit theorem \cite[Proposition 4.9.]{gps10} (stating $\al_4^\e(\delta)/ \al_4^\e(1)\to\delta^{-5/4}$) instead of quasi-multiplicativity, we could have concluded on the hexagonal lattice that
 $$ P^\la_\e \big[\VG(t,\g_\e) \mid \g_\e[0,\s]\big] \,-\, P^\mu_\e \big[\VG(t,\g_\e) \mid \g_\e[0,\s]\big]
 \,\ge\, C \delta^\frac{3}{4}$$
 for small enough $\delta$ and $\e\ll\delta$ on $G(t,\g_\e)$, for some constant $C>0$ independent of $\e$ and $\delta$.
\end{rem}
\begin{rem}
 Though the proof of Lemma~\ref{lem:evg-disc} is almost the same as the corresponding part of \cite{nw9}, it contains the main reason, why \cite[Proposition 6]{nw9} expands to Theorem~\ref{mainthm}: it is the quite trivial equation~(\ref{eq:pla-pmu}). This equation shows that the distance between two different nearcritical probabilities is -- up to constants -- the same as the distance between a nearcritical and the critical probability. In fact,
 $$ p^\la\,-\,p^\mu \quad\asymp\quad \frac{\e^2}{\al_4^\e(1)} \quad\asymp\quad p_\text{nearcritical}\,-\,p_\text{critical} $$
 as $\e\to0$.
\end{rem}

\subsection{Many Mesoscopic Triangles} \label{ssec:manytriangles}

We continue the proof of Theorem~\ref{mainthm} similar to \cite[p.\ 816]{nw9} by looking at a whole bunch of small triangles. Thereto let $\delta\gg\e>0$. Later on we will send $\e$ -- and finally even $\delta$ -- to zero, but in this subsection $\delta$ and $\e$ are fixed. Using a triangular grid of mesh size $4\delta$, we place a circle of radius $\delta$ at each site and put an equilateral triangle of size $\delta$ in its centre. This defines $N=N(\delta)\asymp\delta^{-2}$ deterministic triangles on the whole domain. We fix some very small $\beta>0$ and set $M=M(\delta):=\lfloor \delta^{-2+\at+\beta} \rfloor$, where $\at$ is the exponent bounding the two arm probability from below.

Given the discrete exploration path $\g_\e$, we assign each triangle $t$ its hitting time $\s(t,\g_\e)$ as defined at the beginning of the proof. If a triangle is not hit at all, we set $\s(t,\g_\e)=1$. We arrange the $N$ triangles in the order $t_1,\ldots,t_N$ such that $\s_1\le\s_2\le\ldots\le\s_N$ where $\s_k=\s(t_k,\g_\e)$. Note that these inequalities are strict unless $\s_k=\s_{k+1}=1$. We further introduce the $\s$-Algebras on $\Om_\e$
$$ \F_k:=\s(\g_\e[0,\s_{k+1}])\,,\quad k\in\{0,\ldots,N-1\} $$
and $\F_N=\F_{N+1}:=\s(\g_\e[0,1])$. Note the shift in the index and the very different meaning of the two letters $\s$ in that formula. Let us remark that we can already decide at time $\s_k$ whether the triangle $t_k$ is good or not, i.e.\ $G(t_k,\g_\e)\in\F_{k-1}$. Moreover, $\VG(t_k,\g_\e)\in\F_k$ since if $t_k$ is good, the status \emph{very good} is decided at the next hitting of the triangle's boundary and thus before hitting the next triangle at time $\s_{k+1}$. 

Instead of defining a random variable which resembles the quantity $Z$ of \cite[p.\ 817]{nw9} right now, we develop a discrete analogue. With that approach we can explicitly estimate some variances. To this end, we define the bounded random variables $\Om_\e\to\R$
$$ X^{\delta,\io}_{\e,n} := \sum_{k=1}^n \1_{G(t_k,\g_\e)}\big( \1_{\VG(t_k,\g_\e)}-P^\io_\e[\VG(t_k,\g_\e)\mid\F_{k-1}] \big) $$
for $n\in\{0,\ldots,N\}$ and $\io\in\{\mu,\la\}$. Moreover, $X^{\delta,\io}_{\e,N+1}:=X^{\delta,\io}_{\e,N}$. By the remark in the previous paragraph, $X^{\delta,\io}_{\e,n}$ is $\F_n$-measurable. In fact, it is a martingale with respect to $P^\io_\e$ since
\begin{multline*}
 E_{P^\io_\e}\big[\1_{G(t_n,\g_\e)}\big( \1_{\VG(t_n,\g_\e)}-P^\io_\e[\VG(t_n,\g_\e)\mid\F_{n-1}] \big)\mid\F_{n-1}\big] = \\
 = \1_{G(t_n,\g_\e)}\big( E_{P^\io_\e}[\1_{\VG(t_n,\g_\e)}\mid\F_{n-1}]-P^\io_\e[\VG(t_n,\g_\e)\mid\F_{n-1}] \big) = 0 \,.
\end{multline*}
But we will need a slightly different martingale. To this end, we define for $a\in\N_0$ 
$$ T_a := \inf\big\{n\in\N_0:\sum_{k=1}^n\1_{G(t_k,\g_\e)}\ge a\big\}\wedge (N+1) \,.$$
Then $\{T_a=n\}\in\F_{n-1}$ for all $n\in\{1,\ldots,N+1\}$ and $a\in\N_0$ (with $\F_{-1}:=\{\emptyset,\Om_\e\}$). Thus $T_a$ is a ``pre-visible stopping time", i.e.\ $T_a$ is $\F_{T_a-1}$-measurable. As $(T_a)_{a\in\N_0}$ is a non-decreasing sequence of bounded stopping times, the Optional Sampling Theorem implies that 
$$ \big(X^{\delta,\io}_{\e,T_a}\big)_{a\in\N_0} \text{ is an } (\F_{T_a})_{a\in\N_0}\text{-martingale with respect to }P^\io_\e \,.$$
It follows that
$$ E_{P^\io_\e}\big[X^{\delta,\io}_{\e,T_a}\big] \,=\,0 $$
and
$$ \Var_{P^\io_\e}\big[X^{\delta,\io}_{\e,T_a}\big] \,=\, \sum_{\tilde{a}=0}^{a-1} \Var_{P^\io_\e}\big[X^{\delta,\io}_{\e,T_{\tilde{a}+1}} - X^{\delta,\io}_{\e,T_{\tilde{a}}}\big] \,\le\, \sum_{\tilde{a}=0}^{a-1} 1 \,=\, a $$
since the absolute value of the increments is at most one. Indeed, as $T_a$ counts the number of good triangles, all addends between $T_{\ti{a}}$ and $T_{\ti{a}+1}$ are zero. 

Now we look at the processes stopped at time $T_M$. By Chebyshev's inequality it follows that
$$ P^\mu_\e\big[X^{\delta,\mu}_{\e,T_M} \ge \delta^{-1+\h\at}\big] \le \delta^{2-\at}\Var_{P^\mu_\e}[X^{\delta,\mu}_{\e,T_M}\big] \le \delta^{2-\at} \cdot M \le \delta^{2-\at}\cdot\delta^{-2+\at+\beta}=\delta^\beta .$$
Moreover, we have by Lemma~\ref{lem:evg-disc} on the event that there are at least $M$ good triangles, i.e.\ on $\{T_M\le N\}$
\begin{eqnarray*}
 X^{\delta,\mu}_{\e,T_M}
 &=&    X^{\delta,\la}_{\e,T_M} + \sum_{k=1}^{T_M} \1_{G(t_k,\g_\e)}\big(P^\la_\e[\VG(t_k,\g_\e)\mid\F_{k-1}]-P^\mu_\e[\VG(t_k,\g_\e)\mid\F_{k-1}]\big)\\
 &\ge&  X^{\delta,\la}_{\e,T_M} + \sum_{k=1}^{T_M} \1_{G(t_k,\g_\e)} \cdot \delta^{2-\af+\frac\beta2} \\
 &=&    X^{\delta,\la}_{\e,T_M} + M \cdot \delta^{2-\af+\frac\beta2} 
 \ge    X^{\delta,\la}_{\e,T_M} + \delta^{\at-\af+2\beta}
\end{eqnarray*}
for small enough $\delta$ and $\e\ll\delta$. Therefore
\begin{eqnarray*}
 P^\la_\e\big[X^{\delta,\mu}_{\e,T_M} \le \2\delta^{\at-\af+2\beta}, T_M\le N\big]
 &\le& P^\la_\e\big[X^{\delta,\la}_{\e,T_M} + \delta^{\at-\af+2\beta}\le \2\delta^{\at-\af+2\beta}\big] \\
 &=&   P^\la_\e\big[X^{\delta,\la}_{\e,T_M} \le -\2\delta^{\at-\af+2\beta}\big] \\
 &\le& 4 \,\delta^{2\af-2\at-4\beta}\,\Var_{P^\la_\e}[X^{\delta,\la}_{\e,T_M}]\\
 &\le& 4 \,\delta^{2\af-2\at-4\beta}\cdot M \,\le\, 4 \,\delta^{2\af-\at-2-3\beta}\,.
\end{eqnarray*}
Thus we arrived at
\begin{lem} \label{lem:est_disc}
 The following estimates hold:
 $$  P^\mu_\e\big[X^{\delta,\mu}_{\e,T_M} \ge \delta^{-1+\h\at}\big] \;\le\; \delta^\beta $$
 whereas
 $$ P^\la_\e\big[X^{\delta,\mu}_{\e,T_M} \le \2\delta^{\at-\af+2\beta}\big] \;\le\; 4\, \delta^{2\af-\at-2-3\beta} + P^\la_\e\big[T_M=N+1\big] $$
 for all small enough $\delta$ and $\e\ll\delta$. \qed
\end{lem}
Let us remark that $2\af-\at-2-3\beta>0$ for small $\beta$ by inequality (\ref{eq:inexp}). The main ingredient to Lemma~\ref{lem:est_disc} was the estimate of the variance. For $\io\in\{\mu,\la\}$, we estimated the variance of $X^{\delta,\io}_{\e,T_M}$ with respect to $P^\io_\e$ using a martingale structure. But this approach did not yield an estimate of the variance of $X^{\delta,\mu}_{\e,T_M}$ with respect to $P^\la_\e$ (mind the $\la$ and the $\mu$), which, together with the corresponding expectation, would have been nice for the second statement of Lemma~\ref{lem:est_disc}. Instead we used a point-wise estimate of $X^{\delta,\mu}_{\e,T_M}-X^{\delta,\la}_{\e,T_M}$.

One could be tempted to simply estimate the variance by independence since the considered triangles are disjoint. But this account is tricky since the exploration path obviously depends on its past, and therefore the events $G(t_k,\g_\e)$ respectively $\VG(t_k,\g_\e)$, $k\in\{1,\ldots,N\}$, are not independent. Moreover, the exploration path could enter, leave and re-enter the bottom half of the rectangle $r$ of some triangle while making a different triangle good and possibly very good in the meantime. Hence we chose the martingale approach described above which does not use any geometric information. Alternatively, it may be possible to estimate the variance with some ideas used in the proof of Lemma~\ref{lem:numgoodtrian} below.

We still have to look at the event $\{T_M=N+1\}$, i.e.\ at the event that there are less than $M=\lfloor \delta^{-2+\at+\beta} \rfloor$ good triangles to benefit from the second estimate of Lemma~\ref{lem:est_disc}.
\begin{lem} \label{lem:numgoodtrian}
 There are a function $J$ with $J(\delta)\to\infty$ as $\delta\to0$ and a numerical constant $C_0\in(0,1)$  such that 
 $$  P^\la_\e\big[T_M=N+1\big] \,=\, P^\la_\e\big[\sum_{k=1}^{N}\1_{G(t,\g_\e)} < M\big]\;\le\; (1-C_0)^{J(\delta)} $$
 for small enough $\delta$ and $\e\ll\delta$.
\end{lem}
\begin{proof}
 We follow the rough outline in \cite[p.\ 816f]{nw9} using ideas of the proof of \cite[Proposition 2]{nw9}. We choose $J=J(\delta)$ such that $\delta^{\beta} = c_5 (r2^{-J})^{2-\at}$ with some constant $c_5>0$ specified below. The reason for that choice will become clear later on. Since $\at<2$, $J(\delta)$ tends to infinity as $\delta\to0$.
 
 We use the notation $f(\e,\delta,j)\asymp g(\e,\delta,j)$ to indicate that there are numerical constants $c,c'>0$ such that $\exists\,\delta_0>0$ $\forall\,\delta<\delta_0$ $\exists\,\e_0>0$ $\forall\,\e<\e_0$ $\forall\,j\in\{0,\ldots,\lfloor J\rfloor\}$:
 $$ c\,f(\e,\delta,j)\le g(\e,\delta,j) \le c'\,f(\e,\delta,j)\,.$$
 
 Below we will need a statement similar to item (3) in \cite[p.\ 803]{nw9}, namely
 \begin{equation} \label{eq:sum}
  \sum_{k=1}^n k\al^\e_2(\delta,4\delta k) \asymp n^2 \al^\e_2(\delta,4\delta n) \,. 
 \end{equation}
 Since it is harder to cross a larger annulus, we have in one direction
 $$ \sum_{k=1}^n k\al^\e_2(\delta,4\delta k) \ge \sum_{k=1}^n k\al^\e_2(\delta,4\delta n) \ge \2 n^2 \al^\e_2(\delta,4\delta n) \,.$$
 The other direction follows using quasi-multiplicativity:
 $$
  \sum_{k=1}^n k\frac{\al^\e_2(\delta,4\delta k)}{\al^\e_2(\delta,4\delta n)} 
  \le C  \sum_{k=1}^n k\frac{1}{\al^\e_2(4\delta k,4\delta n)}
  \le C \sum_{k=1}^n k \Big(\frac{4\delta n}{4\delta k}\Big)^\at
  \le C \sum_{k=1}^n k \frac{n}{k} = C n^2
 $$
 since $\at\le1$.
 
 Now we begin with the actual proof. We choose $0<\delta<\delta_0$ and $0<\e<\e_0$ for some appropriate $\delta_0>\e_0>0$. Let us recall that our domain is the half-circle $H_r$ with radius $r>0$. We consider the following half-annuli:
 $$ B_j \,:=\, H_{r2^{-j}} \setminus H_{r2^{-j-1}}\qquad j\in\{0,\ldots,\lfloor J\rfloor\}\,, $$
 If $\delta$ and $\beta$ are small enough, then $\delta<2^{-10}c_5^{-1}\delta^\beta=2^{-10}(r2^{-J})^{2-\at}\le r2^{-J-10}$ by the choice of $J$ and $\at\le1$. Thus there are some triangles in the half-annuli. Let $\T_j$ be the set of all triangles which are contained in $B_j$ and whose distance from the boundary of $B_j$ is at least $r2^{-j-3}$. $\T_j$ consists of $\asymp r^22^{-2j}\delta^{-2}$ triangles. For a triangle $t\in\T_j$, let $G_j'(t)$ be the event that there are a blue and a yellow arm originating at $b(t)$, crossing $r(t)$, staying inside $B_j$ and finally ending at the negative respectively positive real axis. If $G_j'(t)$ is fulfilled, then $t$ is good for $\g_\e$, i.e.\ $G_j'(t)\subset G(t,\g_\e)$.
 
 Now we want to estimate the probability of $G_j'(t)$. Note that $G_j'(t)$ implies $A_2(t,\delta,r2^{-j-3})$,  the event that there exist two arms of different colour inside the annulus with radii $\delta$ and $r2^{-j-3}$ centred at the centre of $t$. Conversely if $A_2(t,\delta,r2^{-j-3})$ with some specified separated landing sequences is fulfilled and if some deterministic rectangles of fixed aspect ratio are crossed, then $G_j'(t)$ occurs. By the arm separation lemmas and RSW it follows that
 $$ P^\la_\e\big[G_j'(t)\big] \,\asymp\, P^\la_\e\big[A_2(t,\delta,r2^{-j-3})] \,\asymp\, \al_2^\e(\delta,r2^{-j})\,.$$
  
 Let the random variable $G_j$ be the number of triangles $t\in\T_j$ that fulfil $G_j'(t)$. We want to estimate the probability that $G_j$ is quite small. Thereto we apply the second moment method. While the first moment is immediately estimated:
 $$ E_{P^\la_\e}[G_j] = \sum_{t\in\T_j} P^\la_\e[G_j'(t)] \asymp \sum_{t\in\T_j}\al_2^\e(\delta,r2^{-j}) \asymp r^22^{-2j}\delta^{-2} \al_2^\e(\delta,r2^{-j})\,,$$
 the second moment is more involved. Let $t,\ti{t}\in\T_j$ be two different triangles. Let $\|t;\ti{t}\|$ denote the distance of their centres. If both events $G_j'(t)$ and $G_j'(\ti{t})$ occur, then there are two crossings of different colour in each of the the following three annuli: the annulus around $t$ with radii $\delta$ and $\2\|t;\ti{t}\|$, the annulus around $\ti{t}$ with radii $\delta$ and $\2\|t;\ti{t}\|$, and finally the annulus around the centre between the two triangles with radii $2\|t;\ti{t}\|$ and $r2^{-j-3}$. Since these annuli are disjoint, it follows that
 $$ P^\la_\e[G_j'(t)\cap G_j'(\ti{t})] \,\le\, c_1 \cdot \al^\e_2(\delta,\2\|t;\ti{t}\|)\cdot \al^\e_2(\delta,\2\|t;\ti{t}\|) \cdot \al^\e_2(2\|t;\ti{t}\|,r2^{-j-3})\,.$$ 
 Here and in the following, $c_1,c_2,\ldots,c_7>0$ are numerical constants. Using quasi-multi\-plicativity, we conclude
 \begin{eqnarray*}
  E_{P^\la_\e}\big[G_j^2\big]
  &=&	\sum_{t\in\T_j} P^\la_\e[G_j'(t)] + \sum_{t\ne\ti{t}\in\T_j} P^\la_\e[G_j'(t)\cap G_j'(\ti{t})] \\
  &\le& E_{P^\la_\e}\big[G_j\big] + c_2 \sum_{t\ne\ti{t}\in\T_j} \al^\e_2(\delta,4\delta\big\lfloor\tfrac{1}{8\delta}\|t;\ti{t}\|\big\rfloor)\cdot \al^\e_2(\delta,r2^{-j})\,.
 \end{eqnarray*}
 Since the triangles were placed using a triangular grid of mesh size $4\delta$, there are at most $c_3\cdot k$ triangles in $T_j$ at distance $4\delta k$ from some fixed triangle for $k\in\{1,\ldots,\lfloor r2^{-j}/\delta\rfloor\}$ and no triangles further away. This, equation (\ref{eq:sum}) and the estimate of the first moment imply
 \begin{eqnarray*}
  E_{P^\la_\e}\big[G_j^2\big] -E_{P^\la_\e}\big[G_j\big]
  &\le&  c_2\sum_{t\in\T_j} \sum_{k=1}^{\lfloor r2^{-j}/\delta\rfloor} c_3 \,k\,\al^\e_2(\delta,4\delta\,k)\cdot \al^\e_2(\delta,r2^{-j}) \\
  &\asymp& E_{P^\la_\e}[G_j] \cdot \lfloor r2^{-j}/\delta\rfloor^2 \al^\e_2(\delta,4\delta\lfloor r2^{-j}/\delta\rfloor) \;\asymp\; E_{P^\la_\e}[G_j]^2\,.
 \end{eqnarray*}
 As (note that $\h$ and $M$ will become relevant later on)
 \begin{eqnarray*}
  \2 E_{P^\la_\e}[G_j]
  &\ge& c_4 r^22^{-2j}\delta^{-2} \al_2^\e(\delta,r2^{-j}) 
  \,\ge\, c_5 r^22^{-2j}\delta^{-2} \big(\delta/(r2^{-j})\big)^\at \\
  &=& c_5\big(r2^{-j}\delta^{-1}\big)^{2-\at} 
  \,\ge\,  c_5 \big(r2^{-J}\delta^{-1}\big)^{2-\at} 
  \,=\, \delta^\beta \delta^{-2+\at} \,\ge\, M \,\ge\, 1
 \end{eqnarray*}
 by our choice of $J$ and $M=\lfloor \delta^{-2+\at+\beta} \rfloor$, we conclude
 $$ E_{P^\la_\e}\big[G_j^2\big] \,\le\, c_6\, E_{P^\la_\e}\big[G_j\big]^2+E_{P^\la_\e}\big[G_j\big]
 \,\le\, c_7\,E_{P^\la_\e}\big[G_j\big]^2 \,.$$ 
 Since
 $$ E[X] = E\big[X\1_{X<\h E[X]} + X\1_{X\ge\h E[X]}\big] \le \2 E[X] + E\big[X\1_{X\ge\h E[X]}\big]$$
 and therefore
 $$ \big(\2E[X]\big)^2 \le E\big[X\1_{X\ge\h E[X]}\big]^2 \le E\big[X^2\big] \cdot P\big[X\ge\2E[X]\big] $$
 holds for any non-negative random variable $X$, we conclude
 $$ P^\la_\e\big[G_j\ge \2E_{P^\la_\e}[G_j] \big] \ge \frac{E_{P^\la_\e}[G_j]^2}{4\,E_{P^\la_\e}[G_j^2]} \ge C_0 $$
 for the numerical constant $C_0:=(4c_7)^{-1}\in(0,1)$. 
 
 As $G_j$ depends only on the hexagons inside $B_j$ and as these sets are pairwise disjoint, it follows that
 $$ P^\la_\e\big[G_j < \2E_{P^\la_\e}[G_j] \text{ for all }j\in\{0,\ldots,\lfloor J\rfloor\}\big] \le (1-C_0)^{J+1} \,.$$  
 Now we link the former event to the event of interest to conclude the proof. On the one hand, we have 
 $$ G_j \;\le\; \sum_{k=1}^N\1_{G(t_k,\g_\e)} $$
 for all $j\le J$ since every triangle $t$ with $G_j'(t)$ is good for $\g_\e$. On the other hand, we already estimated for all $j\le J$: 
 $$ \2E_{P^\la_\e}[G_j] \,\ge\, M \,.$$
 Therefore we conclude
 $$ P^\la_\e\big[\sum_{k=1}^N\1_{G(t,\g_\e)} < M\big]\;\le\; P^\la_\e\big[G_j < \2E_{P^\la_\e}[G_j] \text{ for all }j\le J\big] \;\le\; (1-C_0)^{J+1} \,,$$
 which completes the proof.
\end{proof}
In fact, this lemma is the only place where we used the fact that we have a straight boundary near the starting point of the exploration path. Therefore it was possible to define the sets $B_j$ such that the estimates above hold uniformly for all $j$. A smooth boundary would also have been sufficient, but for a fractal boundary additional ideas are necessary. 

\subsection{Continuum Limit} \label{ssec:limit}

Now we want to pass to the limit. Thereto we will need the following convergence lemma. Let us remark, that Nolin and Werner could just rely on Cardy's formula for their convergence results whereas we will have to use Lemma~\ref{lem:kern}.
\begin{lem} \label{lem:convergence}
 Let $\mathcal{T}$ be a finite set of triangles in $H_r$. Then there exists a set $\mathcal{N}\subset\mathcal{S}_r$ with
 $$ \G^\io[\mathcal{N}] = 0\,,\quad \io\in\{\mu,\la\}\,,$$
 and such that for all $\g\in\mathcal{N}^c$ the following holds: If $\g^n$, $n\in\N$, is a sequence in $\mathcal{S}_r$ with $\dist(\g^n,\g)\to0$ as $n\to\infty$, then for all triangles $t\in\mathcal{T}$
 $$ \1_{G(t,\g^n)}\,\to\,\1_{G(t,\g)} \,,\qquad \1_{\VG(t,\g^n)}\,\to\,\1_{\VG(t,\g)} $$
 as $n\to\infty$ and for all $\rho>0$ there exist $n_0\in\N$ and $\e_0>0$ such that for all $n\ge n_0$ and $\e\le\e_0$
 $$ \big| \1_{G(t,\g^n)} P^\io_\e\big[\hc{d(t,\g^n)}\big] - \1_{G(t,\g)} P^\io_\e\big[\hc{d(t,\g)}\big] \big| \le \rho$$
 for $\iota\in\{\mu,\la\}$.
\end{lem}
\begin{proof}
 Let $\mathcal{N}$ be the set of all curves $\g$ which -- for some triangle $t\in\mathcal{T}$ -- hit an end point of $b(t)$ or $a^0(t)$ or only touch $b(t)$ or the boundary of $t'$  without crossing them. Then we claim that RSW implies that $\G^\io[\mathcal{N}]=0$, $\io\in\{\mu,\la\}$. Indeed, considering concentric annuli around any deterministic point yields that $\g$ hits that point with $\G^\io$-probability zero. And if $\g$ touches any deterministic straight line (without crossing it), then there are three macroscopic (i.e.\ of size $r$) arms of alternating colours originating at some point on the line going to one of its sides. Since the $3$-arm half-plane exponent is larger than $1$ (in fact, it is $2$ by RSW considerations, see \cite[Theorem 23]{n7}, for instance), this event has $\G^\io$-probability zero. As $\mathcal{N}$ consists of finitely many such events, the claim follows.
 
 For the remainder of the proof let $\g\in\mathcal{N}^c$, let $\g^n$ converge to $\g$ in the $\dist$-metric and let $t\in\mathcal{T}$.

 Suppose that $t$ is good for $\g$. Since $\dist(\g^n,\g)\to0$, i.e.\ $\g^n[0,1]\to\g[0,1]$ in the Hausdorff sense, and since $\g$ crosses $b$ at $\s$ and does not hit an end point of $b$ (because of $\g\in\mathcal{N}^c$), $t$ is also good for $\g^n$ for all large enough $n$. Conversely, if $t$ is good for $\g^n$ for all large $n$, it is also good for $\g$. Now let $t$ be good for $\g$ and for $\g^n$ for all large $n$. Since $\g$ crosses $\del^0\cup\del^1$ at the first hitting and since $\g$ does not hit $a^0$, the status of being very good is identical for $\g$ and for $\g^n$ for all large enough $n$. Thus we have shown that $\1_{G(t,\g^n)}\to\1_{G(t,\g)}$ and $\1_{\VG(t,\g^n)}\to\1_{\VG(t,\g)}$
 as $n\to\infty$.
 
 For the last assertion let $\rho>0$. We can assume that $t$ is good for $\g$ and for $\g_n$ for all large $n$. Since $d(t,\g)$ is defined as the connected component of $t'\setminus\g[0,\s]$ which contains a point near the tip of $t$ together with some components also defined by $\g[0,\s]$ and as $\dist(\g^n,\g)\to0$, we conclude that $d(t,\g^n)$ converge  in the kernel sense to $d(t,\g)$. Furthermore, $\dist(\g^n,\g)\to0$ implies condition (\ref{eq:regin}). Thus Lemma~\ref{lem:kern} yields that there are $n_0\in\N$ and $\e_0>0$ such that for all $n\ge n_0$ and $\e\le\e_0$
 $$ \Big| P^\io_\e\big[\hc{d(t,\g^n)}\big] - P^\io_\e\big[\hc{d(t,\g)}\big] \Big|
   \:\le\: P^\io_\e\big[\hc{d(t,\g^n)}\,\triangle\,\hc{d(t,\g)}\big] \:\le\: \rho $$
 which implies the last assertion since $G(t,\g^n)$ and $G(t,\g^n)$ for all large $n$ simultaneously hold.
\end{proof}

Inspired by the random variables $T_a$ and $X$ defined on $\Om_\e$, we define the following random variables, but on $\mathcal{S}_r$ this time. We still have $\e\ll\delta$ fixed and we use the triangles defined above. Given a curve $\g\in\mathcal{S}_r$ we arrange them in the order $t_1,\ldots,t_N$ according to their hitting time as above. Recall that $M=\lfloor \delta^{-2+\at+\beta} \rfloor$. We define 
$$ T \::=\: \inf\{n\in\N: \sum_{k=1}^n \1_{G(t_k,\cdot)}\ge M\}\wedge (N+1) $$
and
$$ Z^{\delta,\mu}_\e \::=\: \sum_{k=1}^T \1_{G(t_k,\cdot)}\big(\1_{\VG(t_k,\cdot)}-P^\mu_\e\big[\hc{d(t_k,\cdot)}\big]\big) $$
on $\mathcal{S}_r$. Finally we define, letting $\e\to0$ now,
$$ Z^{\delta,\mu} \::=\: \lim_{\e\to0} Z^{\delta,\mu}_\e \,=\, \sum_{k=1}^T \1_{G(t_k,\cdot)}\big(\1_{\VG(t_k,\cdot)}-\lim_{\e\to0} P^\mu_\e\big[\hc{d(t_k,\cdot)}\big]\big)\,,$$
which resembles the quantity $Z$ in \cite[p.\ 817]{nw9}. The limit exists for all curves $\g\in\mathcal{S}_r$ since we have chosen the subsequence $(\e_k)_{k\in\N}$ with the property that the limit of the crossing probabilities of any quad exists. Note that we defined these random variables only for the parameter $\mu$ and not for $\la$, since we will only need the versions with $\mu$.
\begin{lem} \label{lem:weakconv}
 The laws $Z^{\delta,\mu}_\e (\G^\io_\e)$ converge weakly to $Z^{\delta,\mu} (\G^\io)$ as $\e\to0$, for $\io\in\{\mu,\la\}$.
\end{lem}
\begin{proof}
 Let $\io\in\{\mu,\la\}$. We use Skorokhod's representation theorem to construct the following coupling.  Let $(\bar{\Om},\bar{\A},\bar{P})$ be a suitable probability space and let $\bar{\g}, \bar{\g}_\e: \bar{\Om}\to\mathcal{S}_r$, $\e>0$, random variables such that $ \bar{\g}_\e \to \bar{\g}$ $\bar{P}$-a.s. (in the $\dist$-metric) as $\e\to0$ and such that $\G^\io = \bar{\g}(\bar{P})$ and $\G^\io_\e = \bar{\g}_\e(\bar{P})$, $\e>0$. From Lemma~\ref{lem:convergence} it follows that
 $$ Z^{\delta,\mu}_\e\circ\bar{\g}_\e \;\to\; Z^{\delta,\mu}\circ\bar{\g} \quad \bar{P}\text{-a.s.} $$
 since $\bar{P}\big[\bar{\g}^{-1}[\mathcal{N}^c]\big]=\G^\io[\mathcal{N}^c]=1$ and since every ingredient converges on $\bar{\g}^{-1}[\mathcal{N}^c]$. In particular note that Lemma~\ref{lem:convergence} implies that if we choose any sequence $n(\e)$ such that $n(\e)\to\infty$ as $\e\to0$, then
 $$  \lim_{\e\to0} \1_{G(t,\g^{n(\e)})} P^\mu_\e\big[\hc{d(t,\g^{n(\e)})}\big] \;=\; \lim_{n\to\infty} \lim_{\e\to0} \1_{G(t,\g^n)} P^\mu_\e\big[\hc{d(t,\g^n)}\big] \,,$$
 since the double limit is uniform in $n$ and $\e$. Therefore $\1_{G(t,\bar{\g}_\e)} P^\io_\e[\hc{d(t,\bar{\g}_\e)}]$ converges on $\bar{\g}^{-1}[\mathcal{N}^c]$.
 
 Let $f:\R\to\R$ be a continuous and bounded function. By the Dominated Convergence Theorem, we conclude
 $$ \int f \,d\big(Z^{\delta,\mu}_\e(\G^\io_\e)\big) = \int f(Z^{\delta,\mu}_\e\circ\bar{\g}_\e) \,d\bar{P} \to
    \int f(Z^{\delta,\mu}\circ\bar{\g}) \,d\bar{P} = \int f \,d\big(Z^{\delta,\mu}(\G^\io)\big) $$
 as $\e\to0$. Thus the Portmanteau Theorem yields the desired weak convergence.
\end{proof}
For  that Lemma it is crucial that the limit of $Z^{\delta,\mu}_\e$ does exist, which is ensured by the choice of the sequence $\e_k$ in the very beginning. For the definition of $Z^{\delta,\mu}$, in principle, it is possible to use the limes superior. But then there are problems showing the weak convergence since the sequence used to determine the limes superior may depend on $\g$. The results in \cite{ss11} allowed us to choose the same sequence for all curves.

Now we give the link between the results on the discrete paths and the convergence lemmas to conclude the proof of Theorem~\ref{mainthm}. The key is the following connection between the random variables $X^{\delta,\mu}_{\e,T_M}$, $Z^{\delta,\mu}_\e$ and $\g_\e$. On the event that a triangle $t$ is good for the discrete exploration path $\g_\e$, it is very good if and only if the quad $d(t,\g_\e)$ is crossed. Therefore
$$ X^{\delta,\mu}_{\e,T_M} \:=\: Z^{\delta,\mu}_\e \circ \g_\e $$
by their definitions. We conclude $Z^{\delta,\mu}_\e(\G^\io_\e) = Z^{\delta,\mu}_\e(\g_\e(P^\io_\e)) = (Z^{\delta,\mu}_\e\circ\g_\e)(P^\io_\e) = X^{\delta,\mu}_{\e,T_M}(P^\io_\e)$. Now Lemma~\ref{lem:est_disc} yields
$$ \G^\mu_\e\big[Z^{\delta,\mu}_{\e} \ge \delta^{-1+\h\at}\big] \,=\, P^\mu_\e\big[X^{\delta,\mu}_{\e,T_M} \ge \delta^{-1+\h\at}\big] \;\le\; \delta^\beta $$
and
$$ \G^\la_\e\big[Z^{\delta,\mu}_\e \le \2\delta^{\at-\af+2\beta}\big] \;\le\; 4 \delta^{2\af-\at-2-3\beta} + P^\la_\e\big[T_M=N+1\big]\,. $$
With Lemma~\ref{lem:weakconv} and the Portmanteau Theorem we conclude
$$ \G^\mu\big[Z^{\delta,\mu} > \delta^{-1+\h\at}\big] \,\le\, \liminf_{\e\to0} \G^\mu_\e\big[Z^{\delta,\mu}_{\e} > \delta^{-1+\h\at}\big] \,\le\, \delta^\beta $$
and
\begin{eqnarray*}
 \G^\la\big[Z^{\delta,\mu} < \2\delta^{\at-\af+2\beta}\big] 
 &\le& \liminf_{\e\to0} \G^\la_\e\big[Z^{\delta,\mu}_\e < \2\delta^{\at-\af+2\beta}\big] \\
 &\le& 4 \delta^{2\af-\at-2-3\beta} + (1-C_0)^{J(\delta)}\,,
\end{eqnarray*}
where $J(\delta)$ and $C_0$ are chosen according to Lemma~\ref{lem:numgoodtrian}.
  
Because of inequality (\ref{eq:inexp}), namely $2\af-\at>2$, we can now choose a sequence $\delta_n$, $n\in\N$, converging fast enough to zero such that the bounds on the right hand sides are summable. Then the Borel-Cantelli Lemma yields
$$ \G^\mu[Z^{\delta_n,\mu} > \delta_n^{-1+\h\at}\text{ for infinitely many }n] \,=\, 0 $$
and
$$ \G^\la[Z^{\delta_n,\mu} < \2\delta_n^{\at-\af+2\beta}\text{ for infinitely many }n] \,=\, 0\,. $$
Because of inequality (\ref{eq:inexp}) again, we have $1-\2\at<\af-\at-2\beta$, which implies
$$ \delta^{-1+\h\at} \,<\, \2\delta^{\at-\af+2\beta} $$
for $\delta<1$ small enough. Thus we conclude
$$ \G^\la[Z^{\delta_n,\mu} > \delta_n^{-1+\h\at}\text{ for infinitely many }n] \,=\, 1\,. $$
Therefore we detected an event which has probability zero under $\G^\mu$, but probability one under $\G^\la$. This concludes the proof of Theorem~\ref{mainthm}.

Let us remark that we used inequality (\ref{eq:inexp}) only in the very last paragraph. In fact, this is the only place where we need a property proven only for of the hexagonal lattice, namely the values of two critical exponents.

\section{Consequences for Conformal Maps} \label{sec:conf}

The critical scaling limit is conformally invariant. Does a similar statement hold for nearcritical limits? We can use the result above to give a negative answer to that question. 

Let $D$ be a domain and $f:D\to\ti{D}$ be a conformal map. We consider percolation with $p^\mu_\e=\2+\mu\cdot\e^2/\al_4^\e(1)$ in both domains. Let $a\in\del D$ and $\ti{a}:=f(a)$. We impose some corresponding boundary colours near $a$ and $\ti{a}$. Let $\g_\e$ respectively $\ti{\g}_\e$ be the discrete exploration paths starting at $a$ respectively at $\ti{a}$.  If $\g_\e(P^\mu_\e)\to\G^\mu$ and $\ti{\g}_\e(\ti{P}^\mu_\e)\to\ti{\G}^\mu$ weakly, we consider the following question: How are the laws $f(\G^\mu)$ and $\ti{\G}^\mu$ related? We give an answer in the special case considering a scaling map on $H_r$ for some $r>0$. 

\begin{corol} \label{cor:conf}
 Let $D=H_r$ for some $r>0$ and let $f:D\to\ti{D}$ be the scaling map with factor $\s\in\R^+$, i.e. $f(z)=\s z$. Assume $\g_\e(P^\mu_\e)\to\G^\mu$, $\ti{\g}_\e(\ti{P}^\mu_\e)\to\ti{\G}^\mu$ weakly and that $\ti{P}^\mu_\e(\hc{\ti{q}})$ converge as $\e\to0$ for every quad $\ti{q}$ in $\ti{D}$.
 
 If $\s=1$ or $\mu=0$, $f(\G^\mu)$ and $\ti{\G}^\mu$ are identically distributed. But if $\s\ne1$ and $\mu\ne0$, the laws $f(\G^\mu)$ and $\ti{\G}^\mu$ are singular with respect to each other.
\end{corol}
\begin{proof}
 The statement is clear if $\s=1$ since then $f$ is the identity map. If $\mu=0$ we are in the well-known critical case. Thus we may assume $\s\ne1$ and $\mu\ne0$. Let $\om_\e$ be a realization of percolation in $D$ with mesh size $\e$ and $p^\mu_\e=\2+\mu\e^2/\al_4^\e(1)$. Then $f(\om_\e)$ is a realization of percolation in $\ti{D}$ with mesh size $\s\e=:\zeta$. Each hexagon of $f(\om_\e)$ is blue with probability
  $$p_{\zeta}' =  \h+\mu\frac{\e^2}{\al_4^\e(1)} 
  = \h+\mu\frac{\al_4^{\s\e}(1)}{\s^2\al_4^\e(1)} \cdot\frac{(\s\e)^2}{\al_4^{\s\e}(1)} 
  = \h+\mu\s^{\frac{5}{4}}(1+o(1))\s^{-2} \cdot\frac{\zeta^2}{\al_4^{\zeta}(1)}\,, $$
 where we used $\al_4^{\s\e}(1)=\al_4^{\e}(\s^{-1})$ and the ratio limit theorem \cite[Proposition 4.9.]{gps10} stating $\lim_{\e\to0}\al_4^\e(\delta)/\al_4^\e(1)=\delta^{-5/4}$. Therefore $f(\om_\e)$ is a realization of percolation in $\ti{D}$ with mesh size $\zeta$ and $p_\zeta^\la=\2+\la_\zeta\cdot\zeta^2/\al_4^{\zeta}(1)$ where $\la_\zeta\to\mu\s^{-3/4}=:\la\ne\mu$. Therefore $f(P^\mu_{\zeta/\s})=\ti{P}^\la_\zeta$. Note that $\ti{\g}_\zeta\circ f = f \circ \g_{\zeta/\s}$ by the definition of the exploration paths. Thus $\ti{\g}_\zeta(\ti{P}^\la_\zeta)=\ti{\g}_\zeta(f(P^\mu_{\zeta/\s}))=f(\g_{\zeta/\s}(P^\mu_{\zeta/\s}))$. As we assumed that $\g_\e(P^\mu_\e)$ converge weakly to $\G^\mu$, it follows that $\ti{\g}_\zeta(\ti{P}^\la_\zeta)$ converge weakly to $f(\G^\mu)=:\ti{\G}^\la$ since $f$ is continuous.

 On the other hand, $\ti{\g}_\e$ is the discrete exploration path of percolation in $\ti{D}$ with $p^\mu_\e=\2+\mu\cdot\e^2/\al_4^\e(1)$, whose law converges weakly to $\ti{\G}^\mu$. By Theorem~\ref{mainthm}, $\ti{\G}^\mu$ and $f(\G^\mu)$ are singular with respect to each other (even if $\mu\s^{-\frac34}<\mu$ by the remark in the second paragraph after stating the theorem). 
\end{proof}

\subsection*{Acknowledgement}
The author is grateful to Franz Merkl for stimulating discussions and helpful remarks. This research was supported by a scholarship of the Cusanuswerk, one of the German national academic foundations.

\end{document}